\documentclass[12pt]{article}

\usepackage{bm}
\usepackage{tikz}
\usepackage{pgfplots} 
\usepackage{amsfonts}
\usepackage{amsmath, amssymb, amsthm}

\newtheorem{theorem}{Theorem}

\newtheorem{lemma}[theorem]{Lemma}

\newtheorem{proposition}[theorem]{Proposition}
\newtheorem{remark}[theorem]{Remark}

\begin{document}
\title{Analytic regularity for a singularly perturbed
reaction-convection-diffusion boundary value problem with two small
parameters}


\author{I. Sykopetritou$^{ *}$  and C. Xenophontos
\footnote{Department of Mathematics and Statistics, University of Cyprus, PO BOX 20537, Nicosia 1678, Cyprus.}
}


\maketitle

\begin{abstract}
We consider a second order, two-point, singularly perturbed boundary value
problem, of reaction-convection-diffusion type with two small parameters,
and we obtain analytic regularity results for its solution, under the assumption
of analytic input data. First, we establish
classical differentiability bounds that are explicit in the order of
differentiation and the singular perturbation parameters. Next, for small
values of these parameters we show that the solution can be decomposed into
a smooth part, boundary layers at the two endpoints, and a negligible
remainder. Derivative estimates are obtained for each component of the
solution, which again are explicit in the differentiation order and the
singular perturbation parameters.

\end{abstract}

\section{Introduction}

\label{intro}

Singularly perturbed problems, and the numerical approximation of their
solution, have been studied extensively over the last few decades (see,
e.g., the books \cite{mos}, \cite{morton}, \cite{rst} and the references
therein). As is well known, a main difficulty in these problems is the
presence of \emph{boundary layers} in the solution, which appear due to fact
that the limiting problem (i.e. when the singular perturbation parameter(s)
tend to 0), is of different order than the original one, and the (`extra')
boundary conditions can only be satisfied if the solution varies rapidly in
the vicinity of the boundary -- hence the name {\emph{boundary layers}}.

In most numerical methods, high order derivatives of the solution appear in
the error estimates, hence one should have a clear picture of how these
derivatives grow with respect to the singular perturbation parameter(s). For
low order numerical methods, such as Finite Differences (FD) or the $h$
version of the Finite Element Method (FEM), derivatives up to order 3 are
usually sufficient. For high order methods such as the $hp$ version of the
FEM, derivatives of arbitrary order are needed, thus knowing how these
behave with respect to the singular perturbation parameter(s) as well as
the differentiation order, is necessary. Usually problems of
convection-diffusion or reaction-diffusion type are studied separately and
several researchers have proposed and analyzed numerical schemes for the
robust approximation of their solution (see, e.g., \cite{rst} and the
references therein). When there are two singular perturbation parameters
present in the differential equation, the problem becomes
reaction-convection-diffusion and the relationship between the parameters
determines the `regime' we are in (as shown in Table 1 ahead). In \cite{L},
the numerical solution to this problem was addressed, using the $h$ version
of the FEM as well as appropriate finite differences (see also \cite{BZ}, 
\cite{GORP}, \cite{LR}, \cite{ORPS}, \cite{RU}, \cite{TR}, \cite{TR2}). Our
interest in is high order $hp$ FEM, hence we require information on all
derivatives of the solution. In the present article we obtain information
about the analytic regularity of the solution, using the method of
asymptotic expansions (see also \cite{melenk}), thus providing the tools for
 an $hp$ FEM for the approximation of such problems.

The rest of the paper is organized as follows: in Section \ref{model} we
present the model problem and the regularity of its solution in terms of classical
differentiability. Section \ref{asy} contains the asymptotic expansion for
the solution, under the assumption that the singular perturbation parameters
are small enough. We consider all possible relationships between the
singular perturbation parameters, and establish derivative bounds which are
explicit in the differentiation order as well as the singular perturbation
parameters. We also comment on the transition between the regimes,
in the final subsection of Section \ref{asy}.
Finally, in Section \ref{concl} we summarize our conclusions.

With $I\subset \mathbb{R}$ an open, bounded interval with boundary $\partial I$ and
measure $\left\vert I\right\vert $, we will denote by $C^{k}(I)$ the space
of continuous functions on $I$ with continuous derivatives up to order $k$.
We will use the usual Sobolev spaces $W^{k,m}(I)$ of functions on $I $
with $0,1,2,...,k$ generalized derivatives in $L^{m}\left( I\right) $,
equipped with the norm and seminorm $\left\Vert \cdot \right\Vert _{k,m,I}$
and $\left\vert \cdot \right\vert _{k,m,I}\,$, respectively. When $m=2$, we
will write $H^{k}\left( I\right) $ instead of $W^{k,2}\left( I\right) $, and
for the norm and seminorm, we will write $\left\Vert \cdot \right\Vert
_{k,I} $ and $\left\vert \cdot \right\vert _{k,I}\,$, respectively. The
usual $L^{2}(I)$ inner product will be denoted by $\left\langle \cdot ,\cdot
\right\rangle _{I}$, with the subscript omitted when there is no confusion.
We will also use the space 
\begin{equation*}
H_{0}^{1}\left( I\right) =\left\{ u\in H^{1}\left( I\right) :\left.
u\right\vert _{\partial I }=0\right\} .
\end{equation*}%
The norm of the space $L^{\infty }(I)$ of essentially bounded functions is
denoted by $\Vert \cdot \Vert _{\infty ,I}$. Finally, the notation
\textquotedblleft $a\lesssim b$\textquotedblright\ means \textquotedblleft $%
a\leq Cb$\textquotedblright\ with $C$ being a generic positive constant,
independent of any discretization or singular perturbation parameters.


\section{The model problem and its regularity\label{model}}

We consider the following model problem (cf. \cite{omalley}): Find $u$ such
that 
\begin{eqnarray}
-\varepsilon _{1}u^{\prime \prime }(x)+\varepsilon _{2}b(x)u^{\prime
}(x)+c(x)u(x) &=&f(x)\;,\;x\in I=\left( 0,1\right) ,  \label{de} \\
u(0)=u(1) &=&0\text{ },  \label{bc}
\end{eqnarray}%
where $0<\varepsilon _{1},\varepsilon _{2}\leq 1$ are given parameters that
can approach zero, and the functions $b,c,f$ are given and sufficiently
smooth. In particular, we assume that they are \emph{analytic} functions
satisfying, for some positive constants $\gamma _{f},\gamma _{c},\gamma _{b}$
 independent of $\varepsilon _{1},\varepsilon _{2},$ 
\begin{equation}
\left\Vert f^{(n)}\right\Vert _{\infty ,I}\lesssim n!\gamma
_{f}^{n}\;,\;\left\Vert c^{(n)}\right\Vert _{\infty ,I}\lesssim n!\gamma
_{c}^{n}\;,\;\left\Vert b^{(n)}\right\Vert _{\infty ,I}\lesssim n!\gamma
_{b}^{n} \; \; \forall \;n=0,1,2,...\text{ .}  \label{analytic}
\end{equation}%
In addition, we assume that there exist constants $\beta ,\gamma ,\rho $,
independent of $\varepsilon _{1},\varepsilon _{2},$ such that $\forall
\;x\in \overline{I}=[0,1]$ there holds 
\begin{equation}
b(x)\geq \beta >0\;,\;c(x)\geq \gamma >0\;,\;c(x)-\frac{\varepsilon _{2}}{2}%
b^{\prime }(x)\geq \rho >0.  \label{data}
\end{equation}%
The solution to (\ref{de}), (\ref{bc}) satisfies (see, e.g. \cite{L}) 
\begin{equation}
\left\Vert u\right\Vert _{\infty ,I}\lesssim 1.  \label{Lres}
\end{equation}%
We would like to obtain a similar estimate for $u^{\prime }$. This is
achieved in the following.

\begin{lemma}
\label{lem0} Let $u$ be the solution of (\ref{de}), (\ref{bc}) and assume (%
\ref{analytic}), (\ref{data}) hold. \ Then 
\begin{equation*}
\left\Vert u^{\prime }\right\Vert _{\infty ,I}\lesssim \varepsilon _{1}^{-1}.
\end{equation*}
\end{lemma}

\begin{proof} The proof follows \cite{ms}. Let 
\begin{equation*}
A(x)=\frac{\varepsilon _{2}}{\varepsilon _{1}}\int_{x}^{1}b(t)dt,
\end{equation*}%
and note that $A(1)=0$ and $A^{\prime }(x)=-\frac{\varepsilon _{2}}{%
\varepsilon _{1}}b(x).$ Then, multiplying (\ref{de}) by $e^{A(x)}$ and
integrating from $x$ to $1$, gives%
\begin{equation*}
-\varepsilon _{1}u^{\prime }(1)+\varepsilon _{1}e^{A(x)}u^{\prime
}(x)+\int_{x}^{1}e^{A(t)}c(t)u(t)dt=\int_{x}^{1}e^{A(t)}f(t)dt.
\end{equation*}%
Multiplying by $\varepsilon_1^{-1} e^{-A(x)}$ yields%
\begin{equation}
u^{\prime }(x)=e^{-A(x)}u^{\prime }(1)-\frac{1}{\varepsilon _{1}}%
\int_{x}^{1}e^{A(t)-A(x)}c(t)u(t)dt+\frac{1}{\varepsilon _{1}}%
\int_{x}^{1}e^{A(t)-A(x)}f(t)dt.  \label{eq0L0}
\end{equation}%
Integrating from $0$ to $1,$ we further get%
\begin{equation}
0=u^{\prime }(1)\int_{0}^{1}e^{-A(x)}dx-\frac{1}{\varepsilon _{1}}%
\int_{0}^{1}\int_{x}^{1}e^{A(t)-A(x)}\left[ c(t)u(t)-f(t)\right] dtdx.
\label{eq1L0}
\end{equation}%
Since we wish to first estimate $u^{\prime }(1)$, we need upper and lower
bounds for $\int_{0}^{1}e^{-A(x)}dx$. \ From (\ref{data}) we have%
\begin{equation}
\int_{0}^{1}e^{-A(x)}dx\leq \int_{0}^{1}e^{-\frac{\varepsilon _{2}}{%
\varepsilon _{1}}\beta (1-x)}dx\leq \frac{\varepsilon _{1}}{\varepsilon
_{2}\beta }.  \label{eq2L0}
\end{equation}%
Similarly,%
\begin{equation}
\int_{0}^{1}e^{-A(x)}dx\geq \int_{0}^{1}e^{-\frac{\varepsilon _{2}}{%
\varepsilon _{1}}\left\Vert b\right\Vert _{\infty ,I}(1-x)}dx = \frac{%
\varepsilon _{1}}{\varepsilon _{2}\left\Vert b\right\Vert _{\infty ,I}}%
\left( 1-e^{-\frac{\varepsilon _{2}}{\varepsilon _{1}}\left\Vert
b\right\Vert _{\infty ,I}}\right) .  \label{eq3L0}
\end{equation}%
Also, to estimate the remaining terms in (\ref{eq1L0}), we consider%
\begin{equation*}
\frac{1}{\varepsilon _{1}}\int_{0}^{1}\int_{x}^{1}e^{A(t)-A(x)}dtdx=\frac{1}{%
\varepsilon _{1}}\int_{0}^{1}\int_{x}^{1}e^{A^{\prime }(\zeta )(t-x)}dtdx,
\end{equation*}%
for some $\zeta $ between $t$ and $x$. Hence,%
\begin{equation*}
\frac{1}{\varepsilon _{1}}\int_{0}^{1}\int_{x}^{1}e^{A(t)-A(x)}dtdx\leq 
\frac{1}{\varepsilon _{1}}\int_{0}^{1}\int_{x}^{1}e^{-\frac{\varepsilon _{2}%
}{\varepsilon _{1}}\beta (t-x)}dtdx\lesssim \frac{1}{\varepsilon _{2}}+\frac{%
\varepsilon _{1}}{\varepsilon _{2}^{2}}.
\end{equation*}%
Using (\ref{eq1L0})--(\ref{eq3L0}), we get%
\begin{eqnarray*}
\left\vert u^{\prime }(1)\right\vert &\lesssim &\frac{1}{\left[
\int_{0}^{1}e^{-A(x)}dx\right] }\left[ \left( \left\Vert c\right\Vert
_{\infty ,I}\left\Vert u\right\Vert _{\infty ,I}+\left\Vert f\right\Vert
_{\infty ,I}\right) \left( \frac{1}{\varepsilon _{2}}+\frac{\varepsilon _{1}%
}{\varepsilon _{2}^{2}}\right) \right] \\
&\lesssim &\varepsilon _{2}\frac{\left\Vert b\right\Vert _{\infty ,I}}{%
\varepsilon _{1}}\left( 1-e^{-\frac{\varepsilon _{2}}{\varepsilon _{1}}%
\left\Vert b\right\Vert _{\infty ,I}}\right) ^{-1}\left( \frac{1}{%
\varepsilon _{2}}+\frac{\varepsilon _{1}}{\varepsilon _{2}^{2}}\right) \\
&\lesssim &\varepsilon _{1}^{-1}.
\end{eqnarray*}%
Inserting this bound in (\ref{eq0L0}) gives
\begin{eqnarray*}
\left\vert u^{\prime }(x)\right\vert &\lesssim &\varepsilon _{1}^{-1} +\frac{1}{\varepsilon _{1}}\left(
\left\Vert c\right\Vert _{\infty ,I}\left\Vert u\right\Vert _{\infty
,I}+\left\Vert f\right\Vert _{\infty ,I}\right) \int_{x}^{1}e^{A(t)-A(x)}dt \\
&\lesssim & \varepsilon _{1}^{-1} +\frac{1}{\varepsilon _{1}}\left( \left\Vert c\right\Vert _{\infty ,I}\left\Vert
u\right\Vert _{\infty ,I}+\left\Vert f\right\Vert _{\infty ,I}\right)
\int_{x}^{1}e^{-\frac{\varepsilon _{2}}{\varepsilon _{1}} \beta (t-x)}dt \\
&\lesssim & \varepsilon _{1}^{-1} +\frac{1}{\varepsilon _{1}}\left(  \frac{\varepsilon_1}{\varepsilon_2 \beta}\right) \\
&\lesssim &\varepsilon _{1}^{-1},
\end{eqnarray*}
as desired. \end{proof}

Using an inductive argument we are able to prove the following.

\begin{theorem}
\label{thm_reg0} Let $u$ be the solution of (\ref{de}), (\ref{bc}), and assume 
$\varepsilon_1 \leq \varepsilon_2$. Then,
there exist positive constants $C, K$, independent of $\varepsilon
_{1},\varepsilon _{2}$ and $u$, such that for $n=0,1,2,...$%
\begin{equation*}
\left\Vert u^{(n)}\right\Vert _{\infty ,I}\leq C K^{n}\max \left\{
n,\varepsilon _{1}^{-1}\right\} ^{n}.
\end{equation*}
\end{theorem}

\begin{proof} The proof is by induction on $n$ and follows \cite{melenk}.
Equation (\ref{Lres}) and Lemma \ref{lem0} give the result for $n=0,1$, so
we assume it holds for $0\leq \nu \leq n+1$ and show that it holds for $n+2$%
.  Differentiating (\ref{de}) $n$ times gives%
\begin{eqnarray*}
-\varepsilon _{1}u^{(n+2)} &=&f^{(n)}-\varepsilon _{2}\left( bu^{\prime
}\right) ^{(n)}-\left( cu\right) ^{(n)} \\
&=&f^{(n)}-\sum_{\nu =0}^{n}\binom{n}{\nu }\left( \varepsilon _{2}b^{(\nu
)}u^{(n+1-\nu )}+c^{(\nu )}u^{(n-\nu )}\right) .
\end{eqnarray*}%
By the induction hypothesis we have
\begin{eqnarray*}
\varepsilon _{1}\left\Vert u^{(n+2)}\right\Vert _{\infty ,I} &\leq & \left\Vert f^{(n)}\right\Vert _{\infty ,I}+ \\
&&+C \sum_{\nu =0}^{n}\binom{n}{\nu }\left[ \varepsilon _{2}\gamma _{b}^{\nu
}\nu !K^{n+1-\nu }\max \left\{ n+1-\nu ,\varepsilon _{1}^{-1}\right\} ^{n+1-\nu }+\right. \\
&&\left. +\gamma _{c}^{\nu }\nu !K^{n-\nu }\max \left\{ n-\nu ,\varepsilon
_{1}^{-1}\right\} ^{n-\nu }\right] .
\end{eqnarray*}
Using the estimates below (which follow by standard considerations) 
\begin{equation*}
\binom{n}{\nu }\nu !\max \left\{ n+1-\nu ,\varepsilon _{1}^{-1}\right\} ^{n+1-\nu }\leq \max \left\{ n+1,\varepsilon
_{1}^{-1}\right\} ^{n+1},
\end{equation*}%
\begin{equation*}
\binom{n}{\nu }\nu !\max \left\{ n-\nu ,\varepsilon _{1}^{-1}\right\} ^{n-\nu }\leq \max \left\{ n+1,\varepsilon
_{1}^{-1}\right\} ^{n+1},
\end{equation*}%
\begin{equation*}
\left\Vert f^{(n)}\right\Vert _{\infty ,I}\leq C \gamma _{f}^{n}n!\leq C \gamma _{f}^{n}
\max \left\{ n+1,\varepsilon _{1}^{-1}\right\} ^{n+1},
\end{equation*}%
we obtain
\begin{equation*}
\varepsilon _{1}\left\Vert u^{(n+2)}\right\Vert _{\infty ,I}\leq 
C \gamma _{f}^{n}
\max \left\{ n+1,\varepsilon _{1}^{-1}\right\} ^{n+1}+
\end{equation*}%
\begin{equation*}
+C K^{n+2}\max \left\{ n+1,\varepsilon _{1}^{-1}\right\}
^{n+1}\sum_{\nu =0}^{n}\left[ \frac{1}{K}\left( \frac{\gamma _{b}}{K}\right)
^{\nu }+\frac{1}{K^{2}}\left( \frac{\gamma _{c}}{K}\right) ^{\nu }\right]
\end{equation*}%
\begin{equation*}
\leq C K^{n+2} \max \left\{ n+1,\varepsilon _{1}^{-1}\right\} ^{n+1}\left[ \frac{1}{K^{2}}+\frac{1}{K}\frac{1}{(1-\gamma
_{b}/K)}+\frac{1}{K^{2}}\frac{1}{(1-\gamma _{c}/K)}\right] ,
\end{equation*}%
where we choose the constant $K>\max \{1,\gamma _{f},\gamma _{b},\gamma _{c}\}$ such
that the expression in brackets above is bounded by 1. Thus%
\begin{equation}
\varepsilon _{1}\left\Vert u^{(n+2)}\right\Vert _{\infty ,I}\leq C K^{n+2}\max \left\{ n+1,\varepsilon _{1}^{-1}\right\}^{n+1},  \label{thm_reg0eq0}
\end{equation}
and dividing by $\varepsilon _{1}$ gives the desired result. \end{proof}

\begin{remark}
The above result only treats the case $\varepsilon_1 \leq \varepsilon_2$, since if $\varepsilon_2$ is much smaller than 
$\varepsilon_1$, then we have a `regular perturbation' of reaction-diffusion type. If one
considers the limiting case $\varepsilon _{2}=0$, then one sees that there
are two boundary layers, one at each endpoint, of width $O\left( \varepsilon_{1}^{1/2} \right) $. Hence,
the result of Theorem \ref{thm_reg0} should read 
\[
\left\Vert u^{(n)}\right\Vert _{\infty ,I}\leq C K^{n}\max \left\{n,\varepsilon _{1}^{-1/2}\right\} ^{n}.
\]
\end{remark}

More details arise if one studies the structure of the solution to (\ref{de}%
), which depends on the roots of the characteristic equation associated with
the differential operator. \ For this reason, we let $\lambda
_{0}(x),\lambda _{1}(x)$ be the solutions of the characteristic equation and
set%
\begin{equation}
\mu _{0}=-\underset{x\in \lbrack 0,1]}{\max }\lambda _{0}(x)\;,\;\mu _{1}=%
\underset{x\in \lbrack 0,1]}{\min }\lambda _{1}(x),  \label{mu}
\end{equation}%
or equivalently,%
\begin{equation*}
\mu _{0,1}=\underset{x\in \lbrack 0,1]}{\min }\frac{\mp \varepsilon _{2}b(x)+%
\sqrt{\varepsilon _{2}^{2}b^{2}(x)+4\varepsilon _{1}c(x)}}{2\varepsilon _{1}}%
,
\end{equation*}
with the minus sign associated with $\mu_0$ and the plus sign with $\mu_1$.
The following hold true \cite{RU,TR}:%
\begin{equation}
\left. 
\begin{array}{c}
1\ll\mu _{0}\leq \mu _{1\;\;},\;\frac{\varepsilon _{2}}{\varepsilon
_{2}+\varepsilon _{1}^{1/2}}\lesssim \varepsilon _{2}\mu _{0}\lesssim
1\;,\;\varepsilon _{1}^{1/2}\mu _{0}\lesssim 1 \\ 
\max \{\mu _{0}^{-1},\varepsilon _{1}\mu _{1}\}\lesssim \varepsilon
_{1}+\varepsilon _{2}^{1/2}\;,\;\varepsilon _{2}\lesssim \varepsilon _{1}\mu
_{1} \\ 
\text{for }\varepsilon _{2}^{2}\geq \varepsilon _{1}:\;\varepsilon
_{1}^{-1/2}\lesssim \mu _{1}\lesssim \varepsilon _{1}^{-1} \\ 
\text{for }\varepsilon _{2}^{2}\leq \varepsilon _{1}:\;\varepsilon
_{1}^{-1/2}\lesssim \mu _{1}\lesssim \varepsilon _{1}^{-1/2}%
\end{array}%
\right\} .  \label{mu_a}
\end{equation}%
The values of $\mu _{0},\mu _{1}$ determine the width of the boundary
layers and since $\left\vert \lambda _{0}(x)\right\vert <\left\vert \lambda
_{1}(x)\right\vert $ the layer at $x=1$ is stronger than the layer at $x=0$.
\ Essentially, there are
 three regimes \cite{L}:

\begin{table}[h]
\begin{center}
\begin{tabular}{||cccc||}
\hline
\label{table1} &  & $\mu _{0}$ & $\mu _{1}$ \\[0.5ex] \hline\hline
convection-diffusion & $\varepsilon _{1}\ll\varepsilon _{2}=1$ & $1$ & $%
\varepsilon _{1}^{-1}$ \\ \hline
convection-reaction-diffusion & $\varepsilon _{1}\ll\varepsilon _{2}^{2}\ll1$
& $\varepsilon _{2}^{-1}$ & $\varepsilon _{2}/\varepsilon _{1}$ \\ \hline
reaction-diffusion & $1\gg\varepsilon _{1}\gg\varepsilon _{2}^{2}$ & $%
\varepsilon _{1}^{-1/2}$ & $\varepsilon _{1}^{-1/2}$ \\[1ex] \hline
\end{tabular}%
\end{center}
\caption{Different regimes based on the relationship between $\protect%
\varepsilon_1$ and $\protect\varepsilon_2$  \cite{L}.}
\end{table}

It was shown in \cite{L} (see also \cite{RU}) that under the assumptions $%
b,c,f\in C^{q}(I)$ for some $q\geq 1$ and $\frac{\varepsilon_2}{2}q\left\Vert b^{\prime
}\right\Vert _{\infty ,I}\lesssim (1-\ell )$ for some $\ell \in (0,1),$ the
solution $u$ to (\ref{de}), (\ref{bc}) can be decomposed into a smooth part $%
S$, a boundary layer part at the left endpoint $E_{0}$ and a boundary layer
part at the right endpoint $E_{1}$, viz. 
\begin{equation}  \label{decomp00}
u=S+E_{0}+E_{1},
\end{equation}
with 
\begin{equation}
\left\vert S^{(n)}(x)\right\vert \lesssim 1\;,\;\left\vert
E_{0}^{(n)}(x)\right\vert \lesssim \mu _{0}^{n}e^{-\ell \mu
_{0}x}\;,\;\left\vert E_{1}^{(n)}(x)\right\vert \lesssim \mu
_{1}^{n}e^{-\ell \mu _{1}(1-x)},  \label{decomp0}
\end{equation}%
for all $x\in \overline{I}$ and for $n=0,1,2,...,q$. This regularity result
is sufficient for proving convergence of a fixed order $h$ FEM, but not for
an $hp$ FEM -- a more refined regularity result is needed for the smooth
part that shows {\emph{how}} the derivatives grow, with respect to the
differentiation order (cf. eq. (\ref{uM1bound}) ahead).

The above considerations suggest the following:
if  $\varepsilon _{1}$ is small compared to $\varepsilon _{2}$,
then it is instructive to consider the limiting
case $\varepsilon _{1}=0$. There is an exponential layer (of width
 $O(\varepsilon _{2})$) at the \emph{left} endpoint. The homogeneous
equation (with constant coefficients) suggests that the different regimes
are $\varepsilon_{1}\ll\varepsilon _{2}^{2},\varepsilon _{1}\approx
\varepsilon _{2}^{2}$ and $\varepsilon _{1}\gg\varepsilon _{2}^{2}$, as
discussed below:
\begin{enumerate}
\item In the regime $\varepsilon _{1}\ll\varepsilon _{2}^{2}$, we have $\mu
_{0}=O(\varepsilon _{2}^{-1})$ and $\mu _{1}=O(\varepsilon _{2}\varepsilon
_{1}^{-1})$. Hence $\mu _{1}$ is much larger than $\mu _{0}$ and the
boundary layer in the vicinity of $x=1$ is stronger. Consequently, there
is a layer of width $O(\varepsilon _{2})$ at the left endpoint (the one that
arises from the analysis of the case $\varepsilon _{1}=0$) and additionally,
there is another layer at the right endpoint, of width $O(\varepsilon
_{1}/\varepsilon _{2})$.
\item In the regime $\varepsilon _{1}\approx \varepsilon _{2}^{2}$ there are
layers at both endpoints of width $O(\varepsilon _{2})=O\left( \varepsilon
_{1}^{1/2}\right) $.
\item In the regime $\varepsilon _{2}^{2}\ll\varepsilon _{1}\ll1$, there are
layers at both endpoints of width $O\left( \varepsilon _{1}^{1/2}\right) $.
\end{enumerate}

The above information will be utilized in obtaining regularity estimates for
the solution in all three regimes.

\section{The asymptotic expansion\label{asy}}

We elaborate on (1)--(3) above, and choose an appropriate asymptotic expansion for $u$, in what follows. 

The proofs of each result in the subsequent sections are very similar, hence we will provide the details 
for Section \ref{regime1} and omit certain proofs in Sections \ref{regime2} and \ref{regime3}.

\subsection{The regime $\protect\varepsilon _{1} \ll \protect\varepsilon %
_{2}^{2}\ll1\label{regime1}$}

In this case we anticipate a layer of width $O(\varepsilon _{2})$ at the
left endpoint and a layer of width $O\left( \varepsilon _{1}/\varepsilon
_{2}\right) $ at the right endpoint. To deal with this we define the \emph{%
stretched variables} $\tilde{x}=x/\varepsilon _{2}$ and $\hat{x}%
=(1-x)\varepsilon _{2}/\varepsilon _{1}$, in order for the differentiation
operator to produce the necessary powers of $\varepsilon_1, \varepsilon_2$,
that yield a balanced (in $\varepsilon_1, \varepsilon_2$) equation.

Since we wish to stay along the lines of (\ref{decomp00}), we want the
solution to be comprised of a smooth part (in the slow variable $x$), and
two boundary layers (in the fast variables $\tilde{x},\hat{x}$). Hence, we
make the formal ansatz 
\begin{equation}
u\sim \sum_{i=0}^{\infty }\sum_{j=0}^{\infty }\varepsilon
_{2}^{i}(\varepsilon _{1}/\varepsilon _{2}^{2})^{j}\left( u_{i,j}(x)+\tilde{u%
}_{i,j}^{BL}(\tilde{x})+\hat{u}_{i,j}^{BL}(\hat{x})\right) ,  \label{c1}
\end{equation}%
with $u_{i,j},\tilde{u}_{i,j}^{BL},\hat{u}_{i,j}^{BL}$ to be determined.
Substituting (\ref{c1}) into (\ref{de}), separating the slow and fast
variables, and equating like powers of $\varepsilon _{1}$ and $\varepsilon
_{2}$ , we get (see \cite{Irene} for the details) 
\begin{equation}
\left. 
\begin{array}{c}
u_{0,0}(x)=\frac{f(x)}{c(x)} \\ 
u_{i,0}(x)=-\frac{b(x)}{c(x)}u_{i-1,0}^{\prime }(x),i\geq 1 \\ 
u_{0,j}(x)=u_{1,j}(x)=0,j\geq 1 \\ 
u_{i,j}(x)=\frac{1}{c(x)}\left( u_{i-2,j-1}^{\prime \prime
}(x)-b(x)u_{i-1,j}^{\prime }(x)\right) ,i\geq 2,j\geq 1%
\end{array}%
\right\} ,  \label{c1smooth}
\end{equation}%
\begin{equation}
\left. 
\begin{array}{c}
\tilde{b}_{0}\left( \tilde{u}_{0,0}^{BL}\right) ^{\prime }+\tilde{c}_{0}%
\tilde{u}_{0,0}^{BL}=0 \\ 
\tilde{b}_{0}\left( \tilde{u}_{i,0}^{BL}\right) ^{\prime }+\tilde{c}_{0}%
\tilde{u}_{i,0}^{BL}=-\sum_{k=1}^{i}\left( \tilde{b}_{k}\left( \tilde{u}%
_{i-k,0}^{BL}\right) ^{\prime }+\tilde{c}_{k}\tilde{u}_{i-k,0}^{BL}\right)
,i\geq 1 \\ 
\tilde{b}_{0}\left( \tilde{u}_{0,j}^{BL}\right) ^{\prime }+\tilde{c}_{0}%
\tilde{u}_{0,j}^{BL}=\left( \tilde{u}_{0,j-1}^{BL}\right) ^{\prime \prime
},j\geq 1 \\ 
\tilde{b}_{0}\left( \tilde{u}_{i,j}^{BL}\right) ^{\prime }+\tilde{c}_{0}%
\tilde{u}_{i,j}^{BL}=\left( \tilde{u}_{i,j-1}^{BL}\right) ^{\prime \prime
}-\sum_{k=1}^{i}\left( \tilde{b}_{k}\left( \tilde{u}_{i-k,j}^{BL}\right)
^{\prime }+\tilde{c}_{k}\tilde{u}_{i-k,j}^{BL}\right) ,i\geq 1,j\geq 1%
\end{array}%
\right\} ,  \label{c1BLa}
\end{equation}%
\begin{equation}
\left. 
\begin{array}{c}
\left( \hat{u}_{i,0}^{BL}\right) ^{\prime \prime }+\hat{b}_{0}\left( \hat{u}%
_{i,0}^{BL}\right) ^{\prime }=0,i\geq 0 \\ 
\left( \hat{u}_{0,j}^{BL}\right) ^{\prime \prime }+\hat{b}_{0}\left( \hat{u}%
_{0,j}^{BL}\right) ^{\prime }=\hat{c}_{0}\hat{u}_{0,j-1}^{BL},j\geq 1 \\ 
\left( \hat{u}_{i,1}^{BL}\right) ^{\prime \prime }+\hat{b}_{0}\left( \hat{u}%
_{i,1}^{BL}\right) ^{\prime }=\hat{c}_{0}\hat{u}_{i,0}^{BL}-\hat{b}%
_{1}\left( \hat{u}_{i-1,0}^{BL}\right) ^{\prime },i\geq 1 \\ 
\left( \hat{u}_{1,j}^{BL}\right) ^{\prime \prime }+\hat{b}_{0}\left( \hat{u}%
_{1,j}^{BL}\right) ^{\prime }=\hat{c}_{0}\hat{u}_{1,j-1}^{BL}-\hat{b}%
_{1}\left( \hat{u}_{0,j-1}^{BL}\right) ^{\prime }+\hat{c}_{1}\hat{u}%
_{0,j-2}^{BL},j\geq 2 \\ 
\left( \hat{u}_{i,j}^{BL}\right) ^{\prime \prime }+\hat{b}_{0}\left( \hat{u}%
_{i,j}^{BL}\right) ^{\prime }=\hat{c}_{0}\hat{u}_{i,j-1}^{BL}-\hat{b}%
_{j}\left( \hat{u}_{i-j,0}^{BL}\right) ^{\prime }+ \\ 
\sum_{k=1}^{j-1}\left\{ -\hat{b}_{k}\left( \hat{u}_{i-k,j-k}^{BL}\right)
^{\prime }+\hat{c}_{k}\hat{u}_{i-k,j-k-1}^{BL}\right\} ,i\geq 2,j=2,...,i \\ 
\left( \hat{u}_{i,j}^{BL}\right) ^{\prime \prime }+\hat{b}_{0}\left( \hat{u}%
_{i,j}^{BL}\right) ^{\prime }=\hat{c}_{0}\hat{u}_{i,j-1}^{BL}+ \\ 
\sum_{k=1}^{i}\left\{ -\hat{b}_{k}\left( \hat{u}_{i-k,j-k}^{BL}\right)
^{\prime }+\hat{c}_{k}\hat{u}_{i-k,j-k-1}^{BL}\right\} ,i\geq 2,j>i%
\end{array}%
\right\} ,  \label{c1BLb}
\end{equation}%
where the notation $\tilde{b}_{k}(\tilde{x})=\tilde{x}^{k}b^{(k)}(0)/k!$ , $%
\hat{b}_{k}(\hat{x})=(-1)^{k}\hat{x}^{k}b^{(k)}(1)/k!$ is used, and
analogously for the other terms. (We also adopt the convention that empty
sums are 0.) The BVPs (\ref{c1BLa})--(\ref{c1BLb}) are supplemented with the
following boundary conditions (in order for (\ref{bc}) to be satisfied) for
all $i,j\geq 0$: 
\begin{equation}
\left. 
\begin{array}{c}
\tilde{u}_{i,j}^{BL}(0)=-u_{i,j}(0)\;,\;\lim_{\tilde{x}\rightarrow \infty }%
\tilde{u}_{i,j}^{BL}(\tilde{x})=0 \\ 
\hat{u}_{i,j}^{BL}(0)=-u_{i,j}(1)\;,\;\lim_{\hat{x}\rightarrow \infty }\hat{u%
}_{i,j}^{BL}(\hat{x})=0%
\end{array}%
\right\} .  \label{c1BC}
\end{equation}

Next, we describe the regularity of the functions $u_{i,j},\tilde{u}%
_{i,j}^{BL},\hat{u}_{i,j}^{BL},$ defined by (\ref{c1smooth})--(\ref{c1BC})
above. We begin with $u_{i,j},$ and we have the following.

\begin{lemma}
\label{lemma_c1smooth} Let $u_{i,j}$ be defined by (\ref{c1smooth}) and
assume (\ref{analytic}) holds. Then there exist positive constants $C,K$ and
a complex neighborhood $G$ of $\overline{I}$ such that the complex extension
of $u$ (denoted again by $u$) satisfies
\begin{equation*}
\left\vert u_{i,j}(z)\right\vert \leq C\delta ^{-i}K^{i}i^{i}\,\; \; \forall \;\;
\,z\in G_{\delta }=\left\{ z\in G:dist(z,\partial G)>\delta \right\} .
\end{equation*}
\end{lemma}

\begin{proof} The proof is by induction on $i$. The case $i=0$ holds
trivially, so assume the result holds for $i$ and establish it for $i+1$.
Let $\kappa \in (0,1)$ and let $K>0$ be a constant so that $\left[ \frac{2}{%
K^2} + \frac{1}{K}\right] \leq 1$. We have by (\ref{c1smooth}), the
induction hypothesis with $G_{(1-\kappa )\delta} \supset G_{\delta}$, and
Cauchy's Integral Theorem for Derivatives (we take as contour a circle of radius
$\kappa \delta$ about $z_0 \in G_{\delta}$),

\begin{eqnarray*}
\left\vert u_{i+1,j}(z)\right\vert &\leq & C \left\{ \left\vert u_{i-1,j-1}^{\prime
\prime }(z)\right\vert +\left\vert u_{i,j}^{\prime }(z)\right\vert \right\} \\
&\leq & C \left\{ \frac{2}{(\kappa \delta )^{2}}\left( (1-\kappa )\delta
\right) ^{-i+1}K^{i-1}(i-1)^{i-1}+\frac{1}{(\kappa \delta )}\left( (1-\kappa
)\delta \right) ^{-i}K^{i}i^{i}\right\} \\
&\leq &C \delta ^{-i-1}K^{i+1}(i+1)^{i+1}\left\{ \frac{1}{K^{2}}\frac{1}{%
(i+1)^{2}}\frac{2}{\kappa ^{2}(1-\kappa )^{i-1}}\left( \frac{i-1}{i+1}%
\right) ^{i-1}+\right. \\
&&+\left. \frac{1}{K}\frac{1}{(i+1)}\frac{1}{\kappa (1-\kappa )^{i}}\left( 
\frac{i}{i+1}\right) ^{i}\right\} .
\end{eqnarray*}%
Choose $\kappa =1/(i+1).$ Then we get%
\begin{equation*}
\left\vert u_{i+1,j}(z)\right\vert \leq C \delta ^{-i-1}K^{i+1}(i+1)^{i+1} 
\left[ \frac{2}{K^{2}}+\frac{1}{K}\right] ,
\end{equation*}%
so by the choice of $K$ the expression in brackets is bounded by 1 and this
completes the proof.

\end{proof}

\begin{lemma}
\label{lemma_c1smooth2}Let $u_{i,j}$ be defined by (\ref{c1smooth}) and
assume (\ref{analytic}) holds. Then there exist positive constants $%
K_{1},K_{2}$ such that 
\begin{equation*}
\Vert u_{i,j}^{(n)}\Vert _{\infty ,I}\lesssim
n!K_{1}^{n}i!K_{2}^{i}\,\,\forall \,\,n\in \mathbb{N}.
\end{equation*}
\end{lemma}

\begin{proof} This follows immediately from Lemma \ref{lemma_c1smooth} and
Cauchy's Integral Theorem for derivatives: 
\begin{equation*}
\Vert u_{i,j}^{(n)}\Vert _{\infty ,I}\lesssim \frac{n!}{(n+1)^{n}}\delta
^{-i}K^{i}i^{i}e^{n}\lesssim n!K_{1}^{n}i!K_{2}^{i},
\end{equation*}%
with $K_{1} = e, K_{2} = K/\delta$. \end{proof}

In order to treat the layer terms $\tilde{u}_{i,j}^{BL}, \hat{u}_{i,j}^{BL}$, we will develop some auxiliary results. The following one will be used in the proof of Lemma \ref{lem:claim}, and is an analog of Lemma 7.3.6 in \cite{melenk} (see also Proposition \ref{lemma_aux0} ahead).

\begin{lemma}
\label{lemma_aux} Let $\lambda,\gamma\in \mathbb{C}$ with $Re(\lambda) >0, Re(\gamma)>0$, and let
$\alpha_1, \alpha_2 \in \mathbb{R}^{+}$. Suppose $F$ is  an entire function satisfying, for some $C_F >0$, $i,j\in \mathbb{N}_{0},$ 
\begin{equation*}
\left\vert F(z)\right\vert \leq C_{F}\gamma^{i+j}e^{-Re(\lambda z)} \left(
\alpha_1 i+\alpha_2 j+\left\vert z\right\vert \right) ^{\alpha_1 i+\alpha_2 j}\;\;\forall \;z\in \mathbb{C},
\end{equation*}
and let $v_0 \in \mathbb{C}$. Then, the solution $v:(0,\infty )\rightarrow \mathbb{C}$, 
of the problem
\begin{equation*}
v^{\prime }+\lambda v=F\text{ on }(0,\infty )\;,\;v(0)=v_0,
\end{equation*}
can be extended to an entire function (denoted again by $v$), which
satisfies
\begin{equation*}
\left\vert v(z)\right\vert \leq \left[ \frac{C_{F}}{\vert \lambda \vert} \frac{\gamma^{i+j}}{(\alpha_1 i+ \alpha_2 j+1)}%
\left( \alpha_1 i+\alpha_2 j+\left\vert z\right\vert \right) ^{\alpha_1 i+\alpha_2 j+1}+\left\vert v_0
\right\vert \right] e^{-Re(\lambda z)}\;\;\forall \;z\in \mathbb{C}.
\end{equation*}
\end{lemma}

\begin{proof} Using an integrating factor we find 
\begin{equation*}
v(z)=e^{-\lambda_1 z}\left[ v_0+\int_{0}^{|z|}e^{\lambda s}F(s)ds\right] ,
\end{equation*}
from which we get 
\begin{eqnarray*}
\left\vert v(z)\right\vert &\leq &e^{-Re(\lambda z)}\left[ \left\vert
v_0 \right\vert +\int_{0}^{|z|}\left\vert e^{Re(\lambda s)} F(s)\right\vert ds\right] \\
&\lesssim &e^{-Re(\lambda z)}\left[ \left\vert v_0 \right\vert +%
\frac{C_{F}}{\vert \lambda \vert} \gamma^{i+j} \int_{0}^{|z|}\left\vert \left( \alpha_1 i+\alpha_2 j+\left\vert
s\right\vert \right) ^{\alpha_1 i+\alpha_2 j}\right\vert ds\right] ,
\end{eqnarray*}
where we used the assumption on $F$. The result follows. 
\end{proof}

\begin{lemma}\label{lem:claim}
The functions $\tilde{u}_{i,j}^{BL}$ which satisfy (\ref{c1BLa}), (\ref{c1BC}), are entire and there exist positive constants 
$C, \tilde{\gamma}$ such that
\begin{equation}
\left\vert \left( \tilde{u}_{i,j}^{BL}\right) (z)\right\vert \leq C  \frac{\tilde{\gamma}^{i+j}}{i!}
\left( 2i+j+\left\vert z\right\vert \right) ^{2i+j}e^{-\beta Re(z)}\;,z\in \mathbb{C} \; , \; Re(z)>0,  \label{claim}
\end{equation}
where $\beta = \tilde{c}_0 / \tilde{b}_0$.
\end{lemma}

\begin{proof} We recall that 
$\tilde{b}_{k}(\tilde{x})=\tilde{x}^{k}b^{(k)}(0)/k!$ and $\tilde{c}_{k}(\tilde{x})=\tilde{x}^{k}c^{(k)}(0)/k!$.  
Consequently, there exist positive constants $C_{\tilde{b}}, \gamma_{\tilde{b}}, C_{\tilde{c}}, \gamma_{\tilde{c}}$,
depending solely on $b,c$ such that
\begin{equation}
\label{coeff_fns}
\left\vert \tilde{b}_{k}(z) \right\vert \leq C_{\tilde{b}} \gamma_{\tilde{b}}^k \vert z \vert^k \;\;,\;\;
\left\vert \tilde{c}_{k}(z) \right\vert \leq C_{\tilde{c}} \gamma_{\tilde{c}}^k \vert z \vert^k .
\end{equation}
Then, with $K_2$ the constant from Lemma \ref{lemma_c1smooth2}, and 
$\gamma_{\tilde{b}}, \gamma_{\tilde{c}}$ given by (\ref{coeff_fns}), we choose 
$\tilde{\gamma} > \max\{ K_2, \gamma_{\tilde{b}}, \gamma_{\tilde{c}}\}$ so that
\begin{equation}
\label{gamma}
\left[ \frac{ \gamma_{\tilde{b}} / \tilde{\gamma} }{ 1 - \gamma_{\tilde{b}} / \tilde{\gamma}} 
+ \frac{ \gamma_{\tilde{c}} / \tilde{\gamma} }{ 1 - \gamma_{\tilde{c}} / \tilde{\gamma}}  \right] < 1.
\end{equation}

Next, we note that from (\ref{c1BLa}) we may calculate 
\begin{equation*}
\tilde{u}_{0,0}^{BL}(z)=-u_{0,0}(0)e^{-\frac{\tilde{c}_{0}}{\tilde{b}_{0}}z}.
\end{equation*}
Thus, using Lemma \ref{lemma_c1smooth2} to bound the term $\left\vert
u_{0,0}(0)\right\vert ,$ we get 
\begin{equation*}
\left\vert \tilde{u}_{0,0}^{BL}(z)\right\vert \leq C e^{-\left\vert \frac{%
\tilde{c}_{0}}{\tilde{b}_{0}}z\right\vert }\leq C e^{-\beta Re(z)}, \: \beta =\tilde{c}_{0}/\tilde{b}_{0},
\end{equation*}
thus the claim holds for $i,j=0$. 
For $j=0,i>0,$ we proceed with induction on $i$, while keeping $j$ fixed
at $0$. We have shown the desired result for the case $i=0$, so we assume it
holds for $i>0$ and we will establish it for $i+1$. The function $\tilde{u}_{i+1,0}^{BL}$ satisfies
\begin{equation*}
\tilde{b}_{0}\left( \tilde{u}_{i+1,0}^{BL}\right) ^{\prime }+\tilde{c}_{0}\tilde{u}_{i+1,0}^{BL}=
-\sum_{k=1}^{i+1}\left( \tilde{b}_{k}\left( \tilde{u}_{i+1-k,0}^{BL}\right) ^{\prime }+
\tilde{c}_{k}\tilde{u}_{i+1-k,0}^{BL}\right) =:G_{1},
\end{equation*}
as well as $\tilde{u}_{i+1,0}^{BL}(0)=-u_{i+1,0}(0)$. 
In order to use Lemma \ref{lemma_aux}, we bound the right hand side above as follows:
\begin{eqnarray*}
\left\vert G_1(z) \right\vert &\leq& \sum_{k=1}^{i+1} \left[ \left\vert \tilde{b}_k \right\vert \left\vert\left( \tilde{u}_{i+1-k,0}^{BL}\right) ^{\prime } \right\vert + 
\left\vert  \tilde{c}_k \right\vert \left\vert  \tilde{u}_{i+1-k,0}^{BL} \right\vert \right] \\
&\leq& C \sum_{k=1}^{i+1} \vert z \vert^k \left[  \left( \gamma_{\tilde{b}}^k + \gamma_{\tilde{c}}^k \right) \left\vert  \tilde{u}_{i+1-k,0}^{BL} \right\vert \right],
\end{eqnarray*}
where we used (\ref{coeff_fns}). Cauchy's Integral Theorem for Derivatives and the induction hypothesis yield
\begin{eqnarray*}
\left\vert G_1(z)\right\vert &\leq& 
C \sum_{k=1}^{i+1} e^{-\beta Re(z)} |z|^k \left( \gamma^k_{\tilde{b}} + \gamma^k_{\tilde{c}} \right)\frac{\tilde{\gamma}^{i+1-k}}{(i+1-k)!} \left(2(i+1-k)+|z|\right)^{2(i+1-k)}  
\\
&\leq& C e^{-\beta Re(z)}  \frac{ \tilde{\gamma}^{i+1}}{i!} (2i+|z|)^{2i+1} \sum_{k=1}^{\infty} \left[ \left(
 \frac{\gamma_{\tilde{b}}}{\tilde{\gamma}}\right)^k +  \left(
 \frac{\gamma_{\tilde{c}}}{\tilde{\gamma}}\right)^k \right] \\
&\leq& C e^{-\beta Re(z)}  \frac{ \tilde{\gamma}^{i+1}}{i!} (2i+|z|)^{2i+1},
\end{eqnarray*}
since the geometric series converges to a quantity bounded by 1, by the choice of $\tilde{\gamma}$, see eq. (\ref{gamma}).
Then, Lemma \ref{lemma_aux} yields
\[
\left\vert \tilde{u}_{i+1,0}^{BL} (z) \right\vert \leq C \tilde{\gamma}^{i+1} e^{-\beta Re(z)} 
\left( \frac{ (2( i + 1) +|z| )^{2i+2}}{(i+1)!} + \frac{\vert u_{i+1,0}(0) \vert}{\tilde{\gamma}^{i+1}} \right).
\]
Lemma \ref{lemma_c1smooth2}, the choice of $\tilde{\gamma}$, and Stirling's formula, further give
\begin{eqnarray*}
\left\vert \tilde{u}_{i+1,0}^{BL} (z) \right\vert &\leq& C \tilde{\gamma}^{i+1} e^{-\beta Re(z)} 
\left( \frac{( 2(i+1)+|z| )^{2(i+1)}}{(i+1)!} + (i+1)^{i+1}  \right) \\
&\leq& C \tilde{\gamma}^{i+1} e^{-\beta Re(z)} \frac{( 2(i+1)+|z| )^{2(i+1)} } {(i+1)!} 
\left[ 1 + \frac{ (i+1)^{2(i+1)} }{(2(i+1)+|z| )^{2(i+1)}}  \right] \\ 
&\leq& C \tilde{\gamma}^{i+1} e^{-\beta Re(z)} \frac{( 2(i+1)+|z| )^{2(i+1)} } {(i+1)!} .
\end{eqnarray*}
This completes the induction on $i>0$ (with $j=0$).

We next consider the case $i=0,j>0$. Assuming
\[
\left\vert \tilde{u}_{0,j}^{BL}(z)\right\vert \leq C \tilde{\gamma}^j (j+|z|)^{j} e^{-\beta Re(z)},
\]
we will establish it for $j+1$. The function $ \tilde{u}_{0,j+1}^{BL}(z)$ satisfies for $j>0$,
\begin{equation*}
\tilde{b}_0\left(\tilde{u}_{0, j+1}^{BL}\right)^{\prime}+\tilde{c}_0 \tilde{u}_{0, j+1}^{B L}=\left(\tilde{u}_{0, j}^{B L}\right)^{\prime \prime}.
\end{equation*}
By Cauchy's Integral Theorem for Derivatives and the induction hypothesis, we have
\begin{equation*}
\left\vert \left( \tilde{u}_{0,j}^{BL}(z) \right)'' \right\vert \leq C \tilde{\gamma}^j (j+|z|)^j e^{-\beta Re(z)},
\end{equation*}
and by Lemmata \ref{lemma_c1smooth2}, \ref{lemma_aux},
\begin{eqnarray*}
\left\vert \left( \tilde{u}_{0,j+1}^{BL}(z) \right) \right\vert &\leq& C \tilde{\gamma}^{j} e^{-\beta Re(z)}
\left\{  \frac{(j+|z|)^{j+1}}{j+1} + \frac{\vert u_{0,j+1} \vert}{\tilde{\gamma}^j} \right\} \\
&\leq& C \tilde{\gamma}^{j+1} e^{-\beta Re(z)} (j+1+|z|)^{j+1}  \\
&\times& \left\{  \frac{(j+|z|)^{j+1}}{\tilde{\gamma}(j+1) (j+1+|z|)^{j+1}} + \frac{\tilde{C}}{\tilde{\gamma}^{j+1} (j+1+|z|)^{j+1}} \right\} \\
&\leq& C \tilde{\gamma}^{j+1} e^{-\beta Re(z)} (j+1+|z|)^{j+1}.
\end{eqnarray*}
This establishes the result for $i=0,j>0$.

We finally show the case $i,j>0$. We
perform induction on $i>0$, while keeping $j$ fixed (but arbitrary). We assume (\ref{claim}) holds for $i\geq 1$ 
and show it for $i+1.$ We note that by (\ref{c1BLa}), $\tilde{u}_{i+1,j}^{BL}$ satisfies 
\begin{equation*}
\tilde{b}_{0}\left( \tilde{u}_{i+1,j}^{BL}\right) ^{\prime }+\tilde{c}_{0}%
\tilde{u}_{i+1,j}^{BL}=\left( \tilde{u}_{i+1,j-1}^{BL}\right) ^{\prime \prime
}-\sum_{k=1}^{i+1}\left( \tilde{b}_{k}\left( \tilde{u}_{i+1-k,j}^{BL}\right)
^{\prime }+\tilde{c}_{k}\tilde{u}_{i+1-k,j}^{BL}\right) =:G_{2}\;,\;
\end{equation*}%
as well as $\tilde{u}_{i+1,j}^{BL}(0)=-u_{i+1,j}(0)$. 
We bound $G_2$ using Cauchy's Integral Theorem for Derivatives, (\ref{gamma}), and the induction hypothesis:
\begin{eqnarray*}
\left\vert G_2 \right\vert &\leq& \left\vert \left( \tilde{u}_{i+1,j-1}^{BL}\right) ^{\prime \prime} \right\vert +
\sum_{k=1}^{i+1} \left\{ \left\vert \tilde{b}_{k} \right\vert  \left\vert \left( \tilde{u}_{i+1-k,j}^{BL}\right)
^{\prime } \right\vert + \left\vert \tilde{c}_{k} \right\vert \left\vert  \tilde{u}_{i+1-k,j}^{BL} \right\vert \right\} \\
&\leq& C e^{-\beta Re(z)} \tilde{\gamma}^{i+j} \frac{(2(i+1)+j-1+|z|)^{2(i+1)+j-1}}{(i+1)!} + \\
&+& C e^{-\beta Re(z)} \tilde{\gamma}^{i+1+j}  \sum_{k=1}^{i+1} \frac{\left(2(i+1-k)+j+|z|\right)^{2(i+1)-k+j}}{(i+1-k)!} \left[ \left(
 \frac{\gamma_{\tilde{b}}}{\tilde{\gamma}}\right)^k +  \left(
 \frac{\gamma_{\tilde{c}}}{\tilde{\gamma}}\right)^k \right]   \\
&\leq& C e^{-\beta Re(z)} \tilde{\gamma}^{i+j+1} \frac{(2i+1+j+|z|)^{2i+1+j}}{i!},
\end{eqnarray*}
where we argued in a similar fashion as we did for $G_1$. Lemmata \ref{lemma_c1smooth2}, \ref{lemma_aux} give
\begin{eqnarray*}
\left\vert  \tilde{u}_{i+1,j}^{BL} \right\vert &\leq& C e^{-\beta Re(z)} \tilde{\gamma}^{i+1+j} \left\{
 \frac{(2(i+1)+j+|z|)^{2i+2+j}}{(i+1)!} + \frac{\vert u_{i+1,j}(0) \vert}{\tilde{\gamma}^{i+1+j}} \right\} \\
&\leq& C e^{-\beta Re(z)} \tilde{\gamma}^{i+1+j} \left\{
 \frac{(2(i+1)+j+|z|)^{2(i+1)+j}}{(i+1)!} + \frac{(i+1)^{i+1} }{\tilde{\gamma}^{j}}\right\} \\
&\leq& C  e^{-\beta Re(z)} \tilde{\gamma}^{i+1+j}   \frac{(2(i+1)+j+|z|)^{2(i+1)+j}}{(i+1)!}.
\end{eqnarray*}
This completes the proof.
\end{proof}

For the other layer term $\hat{u}_{i,j}^{BL}$, we have a similar result.

\begin{lemma}\label{lem:claim2}
The functions $\hat{u}_{i,j}^{BL}$ which satisfy (\ref{c1BLb}), (\ref{c1BC}), are entire and there exist positive constants 
$C, \hat{\gamma}$ such that \begin{equation}
\left\vert \left( \hat{u}_{i,j}^{BL}\right) (z)\right\vert \leq C \frac{\hat{\gamma}^{i+j}}{j!}
\left( i+2j+\left\vert z\right\vert \right) ^{i+2j}e^{-\beta Re(z)}\;,z\in \mathbb{C},Re(z)>0  \label{claim2}
\end{equation}
where $\beta = \hat{b}_0$.
\end{lemma}

\begin{proof} The proof is very similar to that of Lemma \ref{lem:claim}, utilizing
Lemmata \ref{lemma_c1smooth2}, \ref{lemma_aux}, and Cauchy's Integral Theorem for Derivatives.
The details appear in \cite{Irene}.
\end{proof}

Using the previous two results, we obtain the following.

\begin{lemma}
\label{lemma_c1BL1} Let the functions $\tilde{u}_{i,j}^{BL},\hat{u}_{i,j}^{BL}$
satisfy (\ref{c1BLa}),(\ref{c1BLb}) respectively. Then, there
exist positive constants $\tilde{C}, \hat{C}, \tilde{K},\hat{K},\tilde{\gamma},\hat{\gamma}$,
depending only on the data, such that $\forall \,\,n\in \mathbb{N},$ 
\begin{equation}
\left\vert \left( \tilde{u}_{i,j}^{BL}\right) ^{(n)}(z)\right\vert \leq \tilde{C} 
\tilde{K}^{n}\tilde{\gamma}^{i+j}\frac{(2i+j)^{2i+j}}{i!} e^{-\tilde{\beta} Re(z)}\;,z\in \mathbb{C%
},Re(z)>0,  \label{Lc1BLa}
\end{equation}%
\begin{equation}
\left\vert \left( \hat{u}_{i,j}^{BL}\right) ^{(n)}(z)\right\vert \leq \hat{C}
\hat{K}^{n}\hat{\gamma}^{i+j}\frac{(i+2j)^{i+2j}}{j!} e^{-\hat{\beta} Re(z)}\;,z\in \mathbb{C}%
,Re(z)>0.  \label{Lc1BLb}
\end{equation}
where $\tilde{\beta} =\tilde{c}_{0}/\tilde{b}_{0}, \hat{\beta} =\hat{b}_{0}$.
\end{lemma}

\begin{proof} We will prove (\ref{Lc1BLa}), since (\ref{Lc1BLb}) is similar.
Cauchy's Integral Theorem for Derivatives allows us
to infer (\ref{Lc1BLa}) from (\ref{claim}) as follows: 
\begin{eqnarray*}
\left\vert \left( \tilde{u}_{i,j}^{BL}\right) ^{(n)}(z)\right\vert 
&\leq & \tilde{C} e^{-\tilde{\beta} Re(z)}\frac{n!}{(n+1)^{n}}\tilde{\gamma}%
^{i+j} \frac{\left( 2i+j+\left\vert z\right\vert \right) ^{2i+j}}{i!}e^{n} \\
&\leq & \tilde{C} e^{-\tilde{\beta} Re(z)}\frac{n!}{(n+1)^{n}}\tilde{\gamma}%
^{i+j} \frac{\left( 2i+j+n \right) ^{2i+j}}{i!}e^{n}.
\end{eqnarray*}%
Observing that
\begin{equation}
\left(2 i+j+n\right) ^{2i+j} = \left(2 i+j\right) ^{2i+j}\left(
1+n/(2i+j)\right) ^{2i+j}\leq \left( 2i+j\right) ^{2i+j}e^{n},  \label{ejei}
\end{equation}%
the result follows.
\end{proof}

We now define, for some $M\in \mathbb{N}$,
\begin{eqnarray}
u_{M}(x) &=&\sum_{i=0}^{M}\sum_{j=0}^{M}\varepsilon _{2}^{i}(\varepsilon
_{1}/\varepsilon _{2}^{2})^{j}u_{i,j}(x),  \label{uM1} \\
\tilde{u}_{M}^{BL}(\tilde{x}) &=&\sum_{i=0}^{M}\sum_{j=0}^{M}\varepsilon
_{2}^{i}(\varepsilon _{1}/\varepsilon _{2}^{2})^{j}\tilde{u}_{i,j}^{BL}(%
\tilde{x}),  \label{uBL1a} \\
\hat{u}_{M}^{BL}(\hat{x}) &=&\sum_{i=0}^{M}\sum_{j=0}^{M}\varepsilon
_{2}^{i}(\varepsilon _{1}/\varepsilon _{2}^{2})^{j}\hat{u}_{i,j}^{BL}(\hat{x}%
),  \label{uBL1b} \\
r_{M} &=&u-\left( u_{M}+\tilde{u}_{M}^{BL}+\hat{u}_{M}^{BL}\right)
\label{rM1}
\end{eqnarray}%
and we have the following decomposition 
\begin{equation}
u=u_{M}+\tilde{u}_{M}^{BL}+\hat{u}_{M}^{BL}+r_{M}.  \label{decomp1}
\end{equation}

As the following theorem shows, the estimates on the smooth part $u_{M}$ in (%
\ref{decomp1}), explicitly show the dependence on the differentiation order.
Moreover, (\ref{uM1bound}) shows that the smooth part is (real) analytic,
hence a high order numerical method could produce exponential rates of
convergence (see, e.g. \cite{melenk}).

\begin{theorem}
\label{thm_reg1}Assume (\ref{analytic}), (\ref{data}) hold, and that 
$\varepsilon _{1} \ll \varepsilon_{2}^{2}$. Then there exist
positive constants $K_{1},K_2,\tilde{K},\hat{K},\tilde{\gamma},\hat{\gamma}$, 
independent of $\varepsilon _{1},\varepsilon _{2},$ such that the
solution $u $ of (\ref{de})--(\ref{bc}) can be decomposed as in (\ref%
{decomp1}), with 
\begin{equation}
\left\Vert u_{M}^{(n)}\right\Vert _{\infty ,I}\lesssim n!K_{1}^{n}\;\forall
\;n\in \mathbb{N}_{0},  \label{uM1bound}
\end{equation}
\begin{equation}
\left\vert \left( \tilde{u}_{M}^{BL}\right) ^{(n)}(x)\right\vert \lesssim 
\tilde{K}^{n}\varepsilon _{2}^{-n}e^{-\beta x/\varepsilon
_{2}}\;\forall \;n\in \mathbb{N}_{0},  \label{uBL1abound}
\end{equation}
\begin{equation}
\left\vert \left( \hat{u}_{M}^{BL}\right) ^{(n)}(x)\right\vert \lesssim \hat{%
K}^{n}\left( \frac{\varepsilon _{1}}{\varepsilon _{2}}\right)
^{-n}e^{-\beta(1-x)\varepsilon _{2}/\varepsilon _{1}}\;\forall
\;n\in \mathbb{N}_{0},  \label{uBL1bbound}
\end{equation}
\begin{equation}
\left\Vert r_{M}\right\Vert _{\infty ,\partial I}+\left\Vert
r_{M}\right\Vert _{0,I}+\varepsilon^{1/2}_{1}\left\Vert r_{M}^{\prime
}\right\Vert _{0,I}\lesssim \max \{e^{-\beta \varepsilon _{2}/\varepsilon
_{1}},e^{-\beta /\varepsilon _{2}}\},  \label{rM1bound}
\end{equation}
where $M$ is chosen so that $\varepsilon _{2} e^2 4 M\max \left\{1, K_{2},\tilde{\gamma},\hat{\gamma%
}\right\} <1$ and $\frac{\varepsilon _{1}}{\varepsilon _{2}^{2}}e^2 4 M\max \{1, 
\tilde{\gamma},\hat{\gamma}\}<1$. The constant $\beta$ is given by 
$\beta = \min \left\{ \frac{\tilde{c}_0}{\tilde{b}_0},  \hat{b}_0 \right\}$.
\end{theorem}
\begin{proof} We first show (\ref{uM1bound}): from (\ref{uM1}) and Lemma \ref%
{lemma_c1smooth2} we have%
\begin{eqnarray*}
\left\Vert u_{M}^{(n)}\right\Vert _{\infty ,I} &\leq
&\sum_{i=0}^{M}\sum_{j=0}^{M}\varepsilon _{2}^{i}(\varepsilon
_{1}/\varepsilon _{2}^{2})^{j}\left\Vert u_{i,j}^{(n)}\right\Vert _{\infty
,I}\lesssim \sum_{i=0}^{M}\sum_{j=0}^{M}\varepsilon _{2}^{i}(\varepsilon
_{1}/\varepsilon _{2}^{2})^{j}n!K_{1}^{n}i!K_{2}^{i} \\
&\lesssim &n!K_{1}^{n}\left( \sum_{i=0}^{M}\varepsilon
_{2}^{i}i^{i}K_{2}^{i}\right) \left( \sum_{j=0}^{M}(\varepsilon
_{1}/\varepsilon _{2}^{2})^{j}\right) \\
&\lesssim &n!K_{1}^{n}\left( \sum_{i=0}^{\infty }\left( \varepsilon
_{2}MK_{2}\right) ^{i}\right) \left( \sum_{j=0}^{\infty }(\varepsilon
_{1}/\varepsilon _{2}^{2})^{j}\right) \\
&\lesssim &n!K_{1}^{n},
\end{eqnarray*}%
since both sums are convergent geometric series due to the assumptions
 $\varepsilon _{2}MK_{2}<1$ and $\varepsilon _{1}/\varepsilon _{2}^{2}<1.$

Next we show (\ref{uBL1abound}): By (\ref{uBL1a}) and Lemma \ref{lemma_c1BL1}, we have
\begin{eqnarray*}
\left\vert \left( \tilde{u}_{M}^{BL}\right) ^{(n)}(\tilde{x})\right\vert
&\leq& \sum_{i=0}^{M}\sum_{j=0}^{M}\varepsilon _{2}^{i}(\varepsilon
_{1}/\varepsilon _{2}^{2})^{j}\left\vert \left( \tilde{u}_{i,j}^{BL}\right)
^{(n)}(\tilde{x})\right\vert \\
&\lesssim& \sum_{i=0}^{M}\sum_{j=0}^{M}\varepsilon _{2}^{i}(\varepsilon
_{1}/\varepsilon _{2}^{2})^{j}\tilde{K}^{n}\tilde{\gamma}%
^{i+j}\frac{(2i+j)^{2i+j}}{i!} e^{- \beta \tilde{x} }.
\end{eqnarray*}
Since $(2i+j)^{2i+j}\leq e^{2i}(2i)^{2i}e^{j}j^{j}$ (cf. (\ref{ejei})), we get 
\begin{eqnarray*}
\left\vert \left( \tilde{u}_{M}^{BL}\right) ^{(n)}(\tilde{x})\right\vert
&\lesssim &\tilde{K}^{n}e^{-\beta \tilde{x}}\left( \sum_{i=0}^{M} \frac{\tilde{%
\gamma}^{i}e^{2i}(2i)^{2i}\varepsilon _{2}^{i} }{i!} \right) \left(
\sum_{j=0}^{M}(\varepsilon _{1}/\varepsilon _{2}^{2})^{j} \tilde{\gamma}^j
e^{j}j^{j}\right) \\
&\lesssim &\tilde{K}^{n}e^{-\beta \tilde{x}}\left( \sum_{i=0}^{\infty
}\left( \tilde{\gamma}e^2 4 M\varepsilon _{2}\right) ^{i}\right) \left(
\sum_{j=0}^{\infty }\left( \frac{\varepsilon _{1}}{\varepsilon _{2}^{2}}%
\tilde{\gamma} e M\right) ^{j}\right) \\
&\lesssim &\tilde{K}^{n}e^{-\beta \tilde{x}},
\end{eqnarray*}%
since both sums are convergent geometric series due to the assumptions $%
4 \tilde{\gamma}e^2M\varepsilon _{2}<1$, $\frac{\varepsilon _{1}}{%
\varepsilon_{2}^{2}} \tilde{\gamma} e^2 4 M<1.$

Similarly, we show (\ref{uBL1bbound}): By (\ref{uBL1b}) and Lemma \ref%
{lemma_c1BL1} 
\begin{eqnarray*}
\left\vert \left( \hat{u}_{M}^{BL}\right)^{(n)}(\hat{x})\right\vert &\leq
&\sum_{i=0}^{M}\sum_{j=0}^{M}\varepsilon _{2}^{i}(\varepsilon
_{1}/\varepsilon _{2}^{2})^{j}\left\vert \left( \hat{u}_{i,j}^{BL}\right)
^{(n)}(\hat{x})\right\vert \\
&\lesssim &\sum_{i=0}^{M}\sum_{j=0}^{M}\varepsilon _{2}^{i}(\varepsilon
_{1}/\varepsilon _{2}^{2})^{j}\hat{K}^{n}\hat{\gamma}^{i+j}\frac{(i+2j)^{i+2j}}{j!}e^{-\beta \hat{x}} \\
&\lesssim &\hat{K}^{n}e^{-\beta \hat{x}}\left( \sum_{i=0}^{\infty }\left( 
\hat{\gamma}e M\varepsilon _{2}\right) ^{i}\right) \left( \sum_{j=0}^{\infty
}\left( \frac{\varepsilon _{1}}{\varepsilon _{2}^{2}} \hat{\gamma} e^2 4 M\right) ^{j}\right) \\
&\lesssim &\hat{K}^{n}e^{-\beta \hat{x}}.
\end{eqnarray*}
It remains to show (\ref{rM1bound}). To this end, note that 
\begin{eqnarray*}
r_{M}(0) &=&u(0)-\left[ \sum_{i=0}^{M}\sum_{j=0}^{M}\varepsilon
_{2}^{i}(\varepsilon _{1}/\varepsilon _{2}^{2})^{j}\left( u_{i,j}(0)+\tilde{u%
}_{i,j}^{BL}(0)+\hat{u}_{i,j}^{BL}(\varepsilon _{2}/\varepsilon _{1})\right) %
\right] \\
&=&-\sum_{i=0}^{M}\sum_{j=0}^{M}\varepsilon _{2}^{i}(\varepsilon
_{1}/\varepsilon _{2}^{2})^{j}\hat{u}_{i,j}^{BL}(\varepsilon
_{2}/\varepsilon _{1}).
\end{eqnarray*}
By (\ref{Lc1BLb}), 
\begin{eqnarray*}
\left\vert r_{M}(0)\right\vert &\leq
&\sum_{i=0}^{M}\sum_{j=0}^{M}\varepsilon _{2}^{i}(\varepsilon
_{1}/\varepsilon _{2}^{2})^{j}\left\vert \hat{u}_{i,j}^{BL}(\varepsilon
_{2}/\varepsilon _{1})\right\vert \\
&\lesssim& \sum_{i=0}^{M}\sum_{j=0}^{M}\varepsilon _{2}^{i}(\varepsilon
_{1}/\varepsilon _{2}^{2})^{j}\hat{\gamma}^{i+j}\frac{(i+2j)^{i+2j}}{j!} e^{-\beta
\varepsilon _{2}/\varepsilon _{1}}\; \\
&\lesssim &e^{-\beta \varepsilon _{2}/\varepsilon _{1}}\left(
\sum_{i=0}^{\infty }\left( \hat{\gamma}Me\varepsilon _{2}\right) ^{i}\right)
\left( \sum_{j=0}^{\infty }\left( (\varepsilon _{1}/\varepsilon _{2}^{2})%
\hat{\gamma} e^2 4 M \right) ^{j}\right) \\
&\lesssim &e^{-\beta \varepsilon _{2}/\varepsilon _{1}}.
\end{eqnarray*}
Similarly, 
\begin{eqnarray*}
\left\vert r_{M}(1)\right\vert &\leq
&\sum_{i=0}^{M}\sum_{j=0}^{M}\varepsilon _{2}^{i}(\varepsilon
_{1}/\varepsilon _{2}^{2})^{j}\left\vert \tilde{u}_{i,j}^{BL}(1/\varepsilon
_{2})\right\vert \\
&\lesssim& \sum_{i=0}^{M}\sum_{j=0}^{M}\varepsilon
_{2}^{i}(\varepsilon _{1}/\varepsilon _{2}^{2})^{j}\tilde{\gamma}%
^{i+j}\frac{(2i+j)^{2i+j}}{i!} e^{-\beta /\varepsilon _{2}}\; \\
&\lesssim &e^{-\beta /\varepsilon _{2}}\left( \sum_{i=0}^{\infty }(\tilde{%
\gamma}e^2\varepsilon _{2} 4 M)^{i}\right) \left( \sum_{j=0}^{\infty }\left(
(\varepsilon _{1}/\varepsilon _{2}^{2}) \tilde{\gamma} e M\right) ^{j}\right)
\\
&\lesssim &e^{-\beta /\varepsilon _{2}}.
\end{eqnarray*}
Combining the two results, we have 
\begin{equation*}
\left\Vert r_{M}\right\Vert _{\infty ,\partial I}\lesssim \max \{e^{-\beta
\varepsilon _{2}/\varepsilon _{1}},e^{-\beta /\varepsilon _{2}}\}.
\end{equation*}%
Now, let $L:=-\varepsilon _{1}\frac{d^{2}}{dx^{2}}+\varepsilon _{2}b\frac{d}{%
dx}+c$ Id, with Id the identity operator, and consider%
\begin{equation*}
L\left( u-u_{M}\right) =f(x)-\sum_{i=0}^{M}\sum_{j=0}^{M}\varepsilon
_{2}^{i}(\varepsilon _{1}/\varepsilon _{2}^{2})^{j}Lu_{i,j}(x),
\end{equation*}%
with $u_{i,j}$ satisfying (\ref{c1smooth}). After some calculations, we find 
\begin{equation*}
L\left( u-u_{M}\right) =-\varepsilon _{2}^{M+1}\sum_{j=1}^{M}\left( \frac{%
\varepsilon _{1}}{\varepsilon _{2}^{2}}\right) ^{j}bu_{M,j}^{\prime },
\end{equation*}
hence
\begin{eqnarray*}
\left\Vert L\left( u-u_{M}\right) \right\Vert _{\infty ,I}&\leq& \varepsilon
_{2}^{M+1}\sum_{j=1}^{M}\left( \frac{\varepsilon _{1}}{\varepsilon _{2}^{2}}%
\right) ^{j}\left\Vert b\right\Vert _{\infty ,I}\left\Vert u_{M,j}^{\prime
}\right\Vert _{\infty ,I} \\
&\lesssim& \varepsilon _{2}^{M+1}\sum_{j=1}^{M}\left( 
\frac{\varepsilon _{1}}{\varepsilon _{2}^{2}}\right) ^{j}\left\Vert
u_{M,j}^{\prime }\right\Vert _{\infty ,I}.
\end{eqnarray*}
Using Lemma \ref{lemma_c1smooth2}, we further obtain%
\begin{equation*}
\left\Vert L\left( u-u_{M}\right) \right\Vert _{\infty ,I}\lesssim
\varepsilon _{2}^{M+1}M!K_{2}^{M}\sum_{j=1}^{M}\left( \frac{\varepsilon _{1}%
}{\varepsilon _{2}^{2}}\right) ^{j}\lesssim \varepsilon _{2}\left(
\varepsilon _{2}MK_{2}\right) ^{M},
\end{equation*}%
since the finite sum can be bounded by a converging geometric series.

We also consider the operator $L$ in the stretched variable $\tilde{x}$, and
we find, after some calculations, 
\begin{eqnarray*}
\tilde{L}\tilde{u}_{M}^{BL} &=&\sum_{i=0}^{M}\sum_{j=0}^{M}\varepsilon
_{2}^{i}(\varepsilon _{1}/\varepsilon _{2}^{2})^{j}\tilde{L}\tilde{u}%
_{i,j}^{BL} \\
&=&\sum_{i=0}^{M}\sum_{j=0}^{M}\varepsilon _{2}^{i}(\varepsilon
_{1}/\varepsilon _{2}^{2})^{j}\left\{ -\varepsilon _{1}\varepsilon
_{2}^{-2}\left( \tilde{u}_{i,j}^{BL}\right) ^{\prime \prime }+\sum_{k=0}^{M}%
\left[ \tilde{b}_{k}\left( \tilde{u}_{i,j}^{BL}\right) ^{\prime }+\tilde{c}%
_{k}\tilde{u}_{i,j}^{BL}\right] \right\}  \\
&=&\left( \frac{\varepsilon _{1}}{\varepsilon _{2}^{2}}\right)
^{M+1}\sum_{i=0}^{M}\varepsilon _{2}^{i}\left( \tilde{u}_{i,M}^{BL}\right)
^{\prime \prime },
\end{eqnarray*}
where (\ref{c1BLa}) was used. Hence, using (\ref{Lc1BLa}), we have 
\begin{eqnarray*}
\left\Vert \tilde{L}\tilde{u}_{M}^{BL}\right\Vert _{\infty ,I} &\leq &\left( 
\frac{\varepsilon _{1}}{\varepsilon _{2}^{2}}\right)
^{M+1}\sum_{i=0}^{M}\varepsilon _{2}^{i}\left\Vert \left( \tilde{u}%
_{i,M}^{BL}\right) ^{\prime \prime }\right\Vert _{\infty ,I} \\
&\lesssim&  \left( 
\frac{\varepsilon _{1}}{\varepsilon _{2}^{2}}\right)
^{M+1}\sum_{i=0}^{M}\varepsilon _{2}^{i}\tilde{\gamma}^{i+M}\frac{(2i+M)^{2i+M}}{i!} \\
&\lesssim &\left( \frac{\varepsilon _{1}}{\varepsilon _{2}^{2}}\right)
^{M+1}\sum_{i=0}^{M}\varepsilon _{2}^{i}\tilde{\gamma}%
^{i+M}e^{2i} (4 i)^{i}e^{M} M^{M} \\
&\lesssim& \left( \frac{\varepsilon _{1}}{%
\varepsilon _{2}^{2}}\tilde{\gamma} e M  \right) ^{M+1}\sum_{i=0}^{M}(\varepsilon _{2}\tilde{%
\gamma}e^2 4M)^{i} \\
&\lesssim &\left( \frac{\varepsilon _{1}}{\varepsilon _{2}^{2}}\tilde{\gamma}%
e M\right) ^{M+1}.
\end{eqnarray*}
Similarly, in the stretched variable $\hat{x}$ we have with the help of (\ref%
{c1BLb}), 
\begin{eqnarray*}
\hat{L}\hat{u}_{M}^{BL} &=&\sum_{i=0}^{M}\sum_{j=0}^{M}\varepsilon
_{2}^{i}(\varepsilon _{1}/\varepsilon _{2}^{2})^{j}\hat{L}\hat{u}_{i,j}^{BL}
\\
&=&\sum_{i=0}^{M}\sum_{j=0}^{M}\varepsilon _{2}^{i}(\varepsilon
_{1}/\varepsilon _{2}^{2})^{j}\left\{ -\frac{\varepsilon _{2}^{2}}{%
\varepsilon _{1}}\left( \hat{u}_{i,j}^{BL}\right) ^{\prime \prime
}-\sum_{k=0}^{M} \left[ \frac{\varepsilon _{2}^{2}}{%
\varepsilon _{1}}\hat{b}_{k}\left( \hat{u}_{i,j}^{BL}\right) ^{\prime }+\hat{%
c}_{k}\hat{u}_{i,j}^{BL}\right] \right\}  \\
&=&\left( \frac{\varepsilon _{1}}{\varepsilon _{2}^{2}}\right)
^{M}\sum_{i=0}^{M}\varepsilon _{2}^{i}\left( \hat{u}_{i,M}^{BL}\right)
^{\prime \prime },
\end{eqnarray*}%
and thus 
\begin{eqnarray*}
\left\Vert \hat{L}\hat{u}_{M}^{BL}\right\Vert _{\infty ,I} &\leq &\left( 
\frac{\varepsilon _{1}}{\varepsilon _{2}^{2}}\right)
^{M}\sum_{i=0}^{M}\varepsilon _{2}^{i}\left\Vert \left( \hat{u}%
_{i,M}^{BL}\right) ^{\prime \prime }\right\Vert _{\infty ,I} \\
&\lesssim& \left( 
\frac{\varepsilon _{1}}{\varepsilon _{2}^{2}}\right)
^{M}\sum_{i=0}^{M}\varepsilon _{2}^{i}\hat{\gamma}^{i+M}\frac{(i+2M)^{i+2M}}{M!} \\
&\lesssim &\left( \frac{\varepsilon _{1}}{\varepsilon _{2}^{2}}\hat{\gamma}%
eM\right) ^{M},
\end{eqnarray*}%
by following the exact same steps as above. Therefore, 
\begin{eqnarray*}
\left\Vert Lr_{M}\right\Vert _{\infty ,I} &=&\left\Vert L\left( u-u_{M}-%
\tilde{u}_{M}^{BL}-\hat{u}_{M}^{BL}\right) \right\Vert _{\infty ,I} \\
&\leq &\left\Vert L\left( u-u_{M}\right) \right\Vert _{\infty ,I}+\left\Vert
\tilde{L}\tilde{u}_{M}^{BL}\right\Vert _{\infty ,I}+\left\Vert \hat{L}\hat{u}%
_{M}^{BL}\right\Vert _{\infty ,I} \\
&\lesssim &\varepsilon _{2}\left( \varepsilon _{2}MK_{2}\right) ^{M}+\left( 
\frac{\varepsilon _{1}}{\varepsilon _{2}^{2}}\tilde{\gamma}eM\right)
^{M+1}+\left( \frac{\varepsilon _{1}}{\varepsilon _{2}^{2}}\hat{\gamma}%
eM\right) ^{M}.
\end{eqnarray*}%
Under the assumptions of the theorem, we have shown that the remainder $r_{M}
$ has exponentially small values at the endpoints of $I,$ and $Lr_{M}$ is
uniformly bounded by an arbitrarily small quantity on $I$. By stability
(see, e.g., \cite{L}) we have the desired result.
\end{proof}

The bounds of the previous theorem are of utmost importance in the design
and proof of convergence (independently of $\varepsilon _{1},\varepsilon
_{2} $) of high order numerical methods, e.g. the $hp$ Finite Element
Method (see, e.g. \cite{SX}). The bounds on the boundary layers tell us \emph{how} to design the
mesh for the approximation, so that the negative powers of $\varepsilon
_{1},\varepsilon _{2}$ are eliminated. The bounds on the smooth part, allow
us to \emph{prove} exponential convergence of the numerical method (see,
e.g., \cite{melenk}).

\subsection{The regime $\varepsilon _{1}\approx \varepsilon_{2}^{2}\label{regime2}$}

Now there are layers at both endpoints of width $O(\varepsilon _{2})$ and
the BVP becomes reaction-diffusion like the one studied in \cite{melenk97}. So with $%
\tilde{x}=x/\varepsilon _{2},\overline{x}=(1-x)/\varepsilon _{2},$ we make,
analogously as in the previous case, the formal ansatz 
\begin{equation}
u\sim \sum_{i=0}^{\infty }\varepsilon _{2}^{i}\left( u_{i}(x)+\tilde{u}%
_{i}^{BL}(\tilde{x})+\overline{u}_{i}^{BL}(\overline{x})\right) ,  \label{c2}
\end{equation}%
with $u_{i},\tilde{u}_{i}^{BL},\overline{u}_{i}^{BL}$ to be determined. \
Substituting (\ref{c2}) into (\ref{de}), separating the slow and
fast variables, and equating like powers of $%
\varepsilon _{1}(\approx \varepsilon _{2}^{2})$ and $\varepsilon _{2}$ we get (see 
\cite{Irene} for details) 
\begin{equation}
\left. 
\begin{array}{c}
u_{0}(x)=\frac{f(x)}{c(x)}\;,\;u_{1}(x)=-\frac{b(x)}{c(x)}u_{0}^{\prime }(x)
\\ 
u_{i}(x)=\frac{1}{c(x)}\left( u_{i-2}^{\prime \prime
}(x)-b(x)u_{i-1}^{\prime }(x)\right) ,i\geq 2%
\end{array}%
\right\} , \label{c2smooth}
\end{equation}%
\begin{equation}
\left. 
\begin{array}{c}
-\left( \tilde{u}_{0}^{BL}\right) ^{\prime \prime }+\tilde{b}_{0}\left( 
\tilde{u}_{0}^{BL}\right) ^{\prime }+\tilde{c}_{0}\tilde{u}_{0}^{BL}=0 \\ 
-\left( \tilde{u}_{i}^{BL}\right) ^{\prime \prime }+\tilde{b}_{0}\left( 
\tilde{u}_{i}^{BL}\right) ^{\prime }+\tilde{c}_{0}\tilde{u}%
_{i}^{BL}=-\sum_{k=1}^{i}\left( \tilde{b}_{k}\left( \tilde{u}%
_{i-k}^{BL}\right) ^{\prime }+\tilde{c}_{k}\tilde{u}_{i-k}^{BL}\right)
,i\geq 1%
\end{array}%
\right\},  \label{c2BLa}
\end{equation}%
\begin{equation}
\left. 
\begin{array}{c}
-\left( \overline{u}_{i}^{BL}\right) ^{\prime \prime }+\bar{b}_{0}\left( 
\overline{u}_{i}^{BL}\right) ^{\prime }+\bar{c}_{0}\overline{u}_{i}^{BL}=0
\\ 
-\left( \overline{u}_{i}^{BL}\right) ^{\prime \prime }+\bar{b}_{0}\left( 
\overline{u}_{i}^{BL}\right) ^{\prime }+\bar{c}_{0}\bar{u}%
_{i}^{BL}=\sum_{k=1}^{i}\left( \bar{b}_{k}\left( \bar{u}_{i-k}^{BL}\right)
^{\prime }-\bar{c}_{k}\bar{u}_{i-k}^{BL}\right) ,i\geq 1%
\end{array}%
\right\},  \label{c2BLb}
\end{equation}%
where the notation $\tilde{b}_{k}(\tilde{x})=\tilde{x}^{k}b^{(k)}(0)/k!$
etc., is used again. The above equations are supplemented with the following
boundary conditions (in order to satisfy (\ref{bc})):

\begin{equation}
\left. 
\begin{array}{c}
u_{i}(0)+\tilde{u}_{i}^{BL}(0)=0 \\ 
u_{i}(1)+\overline{u}_{i}^{BL}(0)=0 \\ 
\lim_{\tilde{x}\rightarrow \infty }\tilde{u}_{i}^{BL}(\tilde{x})=0\,,\,\lim_{%
\overline{x}\rightarrow \infty }\overline{u}_{i}^{BL}(\overline{x})=0%
\end{array}%
\right\}.  \label{c2BC}
\end{equation}
We then define, for some $M\in \mathbb{N},$ 
\begin{equation*}
u_{M}(x)=\sum_{i=0}^{M}\varepsilon _{2}^{i}u_{i}(x),\tilde{u}_{M}^{BL}(%
\tilde{x})=\sum_{i=0}^{M}\varepsilon _{2}^{i}\tilde{u}_{i}^{BL}(\tilde{x}),%
\overline{u}_{M}^{BL}(\overline{x})=\sum_{i=0}^{M}\varepsilon _{2}^{i}%
\overline{u}_{i}^{BL}(\overline{x}),
\end{equation*}
as well as 
\begin{equation}
u=u_{M}+\tilde{u}_{M}^{BL}+\overline{u}_{M}^{BL}+r_{M}.  \label{decomp2}
\end{equation}
We have the following theorem.

\begin{theorem}
\label{thm_reg2}Assume (\ref{analytic}), (\ref{data}) hold, and that 
$\varepsilon _{1}\approx \varepsilon_{2}^{2}$. Then there exist
positive constants $K_{1},K_{2},\tilde{K},\overline{K},\delta $, independent
of $\varepsilon _{1},\varepsilon _{2},$ such that the solution $u$ of (\ref%
{de})--(\ref{bc}) can be decomposed as in (\ref{decomp2}), with 
\begin{equation*}
\left\Vert u_{M}^{(n)}\right\Vert _{\infty ,I}\lesssim n!K_{1}^{n}\;\forall
\;n\in \mathbb{N}_{0},
\end{equation*}%
\begin{equation*}
\left\vert \left( \tilde{u}_{M}^{BL}\right) ^{(n)}(x)\right\vert \lesssim 
\tilde{K}^{n}\varepsilon _{2}^{-n}e^{-\beta x/\varepsilon_{2}}\;\forall \;n\in \mathbb{N}_{0},
\end{equation*}%
\begin{equation*}
\left\vert \left( \overline{u}_{M}^{BL}\right) ^{(n)}(x)\right\vert \lesssim 
\overline{K}^{n}\varepsilon _{2}^{-n}e^{-\beta(1-x)/\varepsilon_{2}}\;\forall \;n\in \mathbb{N}_{0},
\end{equation*}%
\begin{equation*}
\left\Vert r_{M}\right\Vert _{\infty ,\partial I}+\left\Vert
r_{M}\right\Vert _{0,I}+\varepsilon _{2}\left\Vert r_{M}^{\prime
}\right\Vert _{0,I}\lesssim e^{-\delta /\varepsilon _{2}},
\end{equation*}%
where $M$ is chosen so that $\varepsilon _{2}K_{2}M<1$, and $\beta = \min \left\{ \frac{\tilde{b}_0}{\tilde{c}_0}, \frac{\overline{b}_0}{\overline{c}_0} \right\}$.
\end{theorem}
\begin{proof}
When $\varepsilon _{1}\approx \varepsilon_{2}^{2}$, the BVP (\ref{de})--(\ref{bc}) becomes
\begin{eqnarray*}
-u^{\prime \prime }(x)+
\varepsilon^{-1} _{2}b(x)u^{\prime}(x)+\varepsilon^{-2}_2 c(x) u(x) &=&\varepsilon^{-2}_2f(x)\;,\;x\in I=\left( 0,1\right) ,   \\
u(0)=u(1) &=&0.
\end{eqnarray*}
Multiplying the differential equation above by $e^{-\int_0^x \varepsilon^{-1}_2 b(t) dt}$, gives
\begin{equation*}
-\left( u^{\prime }(x) e^{-\int_0^x \varepsilon_2 b(t) dt} \right)' +e^{-\int_0^x \varepsilon^{-1}_2 b(t) dt} c(x)u(x) = e^{-\int_0^x \varepsilon^{-1}_2 b(t) dt} f(x),
\end{equation*}
or equivalently, with $v(x) =  e^{-\int_0^x \varepsilon^{-1}_2 b(t) dt} u(x), F(x) = e^{-\int_0^x \varepsilon^{-1}_2 b(t) dt} f(x)$,
\begin{eqnarray*}
-\varepsilon^{2}_2 v^{\prime \prime }(x)+c(x) v(x) &=&F(x)\;,\;x\in I=\left( 0,1\right) ,   \\
v(0)=v(1) &=&0.
\end{eqnarray*}
The above BVP is in the form considered in \cite{melenk97}, with $c(x)$, $F(x)$ analytic -- the analyticity of
$F(x)$ follows from the analyticity of $b$ and $f$. The desired bounds follow from the results in  \cite{melenk97},
and the fact that $\vert u^{(n)} \vert < \vert v^{(n)} \vert $.
\end{proof}

\subsection{The regime $\varepsilon _{2}^{2}\ll \varepsilon %
_{1}\ll1\label{regime3}$}

We anticipate layers at both endpoints of width $O\left( \sqrt{\varepsilon
_{1}}\right) $. So we define the \emph{stretched variables} $\check{x}=x/%
\sqrt{\varepsilon _{1}}$ and $\grave{x}=(1-x)/\sqrt{\varepsilon _{1}}$ and
make the formal ansatz, analogous to the previous cases, 
\begin{equation}
u\sim \sum_{i=0}^{\infty }\sum_{j=0}^{\infty }\varepsilon _{1}^{i/2}\left(
\varepsilon _{2}/\sqrt{\varepsilon _{1}}\right) ^{j}\left( u_{i,j}(x)+\check{%
u}_{i,j}^{BL}(\check{x})+\grave{u}_{i,j}^{BL}(\grave{x})\right) ,  \label{c3}
\end{equation}%
with $u_{i,j},\check{u}_{i,j}^{BL},\grave{u}_{i,j}^{BL}$ to be determined.
Substituting (\ref{c3}) into (\ref{de}), separating the slow and
fast variables, and equating like powers of $%
\varepsilon _{1}$ and $\varepsilon _{2}$ we get (see \cite{Irene} for the details) 
\begin{equation}
\left. 
\begin{array}{c}
u_{0,0}=\frac{f(x)}{c(x)},u_{1,0}(x)=u_{0,j}(x)=0,j\geq 1 \\ 
u_{i,0}(x)=\frac{1}{c(x)}u_{i-2,0}^{\prime \prime }(x),i\geq 2 \\ 
u_{2i+1,0}(x)=0,i\geq 1 \\ 
u_{1,1}(x)=-\frac{b(x)}{c(x)}u_{0,0}^{\prime }(x),u_{1,j}(x)=0,j\geq 2 \\ 
u_{i,j}(x)=\frac{1}{c(x)}\left( u_{i-2,j}^{\prime \prime
}(x)-b(x)u_{i-1,j-1}^{\prime }(x)\right) ,i\geq 2,j\geq 1%
\end{array}%
\right\} , \label{c3smooth}
\end{equation}%
\begin{equation}
\left. 
\begin{array}{c}
-\left( \check{u}_{0,0}^{BL}\right) ^{\prime \prime }+\check{c}_{0}\check{u}%
_{0,0}^{BL}=0\\ 
-\left( \check{u}_{i,0}^{BL}\right) ^{\prime \prime }+\check{c}_{0}\check{u}%
_{i,0}^{BL}=-\sum_{k=i}^{i}\check{c}_{k}\check{u}_{i-k,0}^{BL},i\geq 1 \\ 
-\left( \check{u}_{0,j}^{BL}\right) ^{\prime \prime }+\check{c}_{0}\check{u}%
_{0,j}^{BL}=-\check{b}_{0}\left( \check{u}_{0,j-1}^{BL}\right) ^{\prime
},j\geq 1 \\ 
-\left( \check{u}_{i,j}^{BL}\right) ^{\prime \prime }+\check{c}_{0}\check{u}%
_{i,j}^{BL}=-\check{b}_{0}\left( \check{u}_{i,j-1}^{BL}\right) ^{\prime }-
\\ 
\sum_{k=1}^{i}\left\{ \check{b}_{k}\left( \check{u}_{i-k,j-1}^{BL}\right)
^{\prime }+\check{c}_{k}\check{u}_{i-k,j}^{BL}\right\} ,i\geq 1,j\geq 1%
\end{array}%
\right\},  \label{c3BLa}
\end{equation}%
\begin{equation}
\left. 
\begin{array}{c}
-\left( \grave{u}_{0,0}^{BL}\right) ^{\prime \prime }+\grave{c}_{0}\grave{u}%
_{0,0}^{BL}=0 \\ 
-\left( \grave{u}_{i,0}^{BL}\right) ^{\prime \prime }+\grave{c}_{0}\grave{u}%
_{i,0}^{BL}=-\sum_{k=1}^{i}\grave{c}_{k}\grave{u}_{i-k,0}^{BL},i\geq 1 \\ 
-\left( \grave{u}_{0,j}^{BL}\right) ^{\prime \prime }+\grave{c}_{0}\grave{u}%
_{0,j}^{BL}=\grave{b}_{0}\grave{u}_{0,j-1}^{BL},j\geq 1 \\ 
-\left( \grave{u}_{i,j}^{BL}\right) ^{\prime \prime }+\grave{c}_{0}\grave{u}%
_{i,j}^{BL}=\left( \grave{b}_{0}\grave{u}_{i,j-1}^{BL}\right) ^{\prime }- \\ 
\sum_{k=1}^{i}\left\{ \grave{b}_{k}\left( \grave{u}_{i-k,j-1}^{BL}\right)
^{\prime }-\grave{c}_{k}\grave{u}_{i-k,j}^{BL}\right\} ,i\geq 1,j\geq 1%
\end{array}%
\right\},  \label{c3BLb}
\end{equation}%
where the notation $\check{b}_{k}(\check{x})=\check{x}^{k}b^{(k)}(0)/k!$
etc., is used once more. The above equations are supplemented with the
following boundary conditions (in order to satisfy (\ref{bc})):

\begin{equation}
\left. 
\begin{array}{c}
\check{u}_{i,j}^{BL}(0)=-u_{i,j}(0) \\ 
\grave{u}_{i,j}^{BL}(0)=-u_{i,j}(1)\\ 
\lim_{\check{x}\rightarrow \infty }\check{u}_{i,j}^{BL}(\check{x}%
)=0\,,\,\lim_{\grave{x}\rightarrow \infty }\grave{u}_{i,j}^{BL}(\grave{x})=0%
\end{array}%
\right\} . \label{c3BC}
\end{equation}%
The following result is established in a completely analogous way as in the
previous cases (cf. Section \ref{regime1}).

\begin{lemma}
\label{lemma_c3smooth}Let $u_{i,j}$ be defined by (\ref{c3smooth}) and
assume (\ref{analytic}) holds. Then there exist positive constants $C,
K_{1},K_{2}$, independent of $\varepsilon_1, \varepsilon_2$, such that 
\begin{equation*}
\Vert u_{i,j}^{(n)}\Vert _{\infty ,I}\leq C
n!K_{1}^{n}i!K_{2}^{i}\,\,\forall \,\,n\in \mathbb{N}.
\end{equation*}
\end{lemma}

Next we consider the boundary layers. The following result was shown in 
\cite{melenk}.

\begin{proposition}[{\protect\cite[Lemma 7.3.6]{melenk}}]
\label{lemma_aux0} Let $\lambda \in \mathbb{C}$ with $Re(\lambda) >0,
Re(\lambda)^{2}>0$. Let $F$ be an entire function satisfying, for some $%
C_{F}>0,$ $j\in \mathbb{N}_{0},$ $q\geq (j+1/2)/|\lambda |>0,$ 
\begin{equation*}
\left\vert F(z)\right\vert \leq C_{F}e^{-Re(\lambda z)}\left( q+\left\vert
z\right\vert \right) ^{j}\;\;\forall \;z\in \mathbb{C}.
\end{equation*}%
Let $\alpha \in \mathbb{C}$ and let $v:(0,\infty )\rightarrow \mathbb{C},$
be the solution of the problem%
\begin{equation*}
-v^{\prime \prime }+\lambda ^{2}v=F\text{ on }(0,\infty
)\;,\;v(0)=g\;,\;\lim_{x\rightarrow \infty }v(x)=0.
\end{equation*}%
Then $v$ can be extended to an entire function (denoted again by $v$), which
satisfies 
\begin{equation*}
\left\vert v(z)\right\vert \leq \left[ C_{F}\frac{1}{|\lambda |}\left(
q+\left\vert z\right\vert \right) ^{j+1}\left( j+1\right) ^{-1}+\left\vert
g\right\vert \right] e^{-Re(\lambda z)}\;\;\forall \;z\in \mathbb{C}.
\end{equation*}
\end{proposition}

Using the above we may prove the following.

\begin{lemma}
\label{lemma_c3BL}Let $\check{u}_{i,j}^{BL},\grave{u}_{i,j}^{BL}$ be defined
by (\ref{c3BLa}), (\ref{c3BLb}), respectively. Then there exist positive
constants $C, \check{\gamma}, \check{K}, \grave{\gamma}, \grave{K}$, depending only on the data, such that $\forall \,\,n\in \mathbb{N},$ 
\begin{equation*}
\left\vert \left( \check{u}_{i,j}^{BL}\right) ^{(n)}(x)\right\vert \leq C \check{\gamma}^{i+j} (i+j)!
\check{K}^{n}\varepsilon _{1}^{-n/2}\,e^{-\beta x/\sqrt{\varepsilon _{1}}%
}\quad \forall \; \; i,j\geq 0,
\end{equation*}%
\begin{equation*}
\left\vert \left( \grave{u}_{i,j}^{BL}\right) ^{(n)}(x)\right\vert \leq C \grave{\gamma}^{i+j} (i+j)!
\grave{K}^{n}\varepsilon _{1}^{-n/2}\,e^{-\beta (1-x)/\sqrt{\varepsilon _{1}}%
}\quad \forall \; \; i,j\geq 0,
\end{equation*}
where $\beta = \min \left\{ \check{c}_{0}, \grave{c}_{0} \right\}$.
\end{lemma}

\begin{proof}
First we show a similar result as (\ref{claim}), by induction on $i,j$ with the help of Proposition \ref{lemma_aux0}. 
For $n>0$, the proof is almost identical to that of Lemma  \ref{lemma_c1BL1}, utilizing Cauchy's Integral Theorem for 
Derivatives. The details appear in \cite{Irene}.
\end{proof}

We then define, for some $M\in \mathbb{N},$%
\begin{eqnarray*}
u_{M}(x) &=&\sum_{i=0}^{M}\sum_{j=0}^{M}\varepsilon _{1}^{i/2}\left(
\varepsilon _{2}/\sqrt{\varepsilon _{1}}\right) ^{j}u_{i,j}(x), \\
\check{u}_{M}^{BL}(\check{x}) &=&\sum_{i=0}^{M}\sum_{j=0}^{M}\varepsilon
_{1}^{i/2}\left( \varepsilon _{2}/\sqrt{\varepsilon _{1}}\right) ^{j}\check{u%
}_{i,j}^{BL}(\check{x}), \\
\grave{u}_{M}^{BL}(\grave{x}) &=&\sum_{i=0}^{M}\sum_{j=0}^{M}\varepsilon
_{1}^{i/2}\left( \varepsilon _{2}/\sqrt{\varepsilon _{1}}\right) ^{j}\grave{u%
}_{i,j}^{BL}(\grave{x}),
\end{eqnarray*}%
and we have the following decomposition:%
\begin{equation}
u=u_{M}+\check{u}_{M}^{BL}+\grave{u}_{M}^{BL}+r_{M}.  \label{decomp3}
\end{equation}%
The theorem that follows is the analog of Theorem \ref{thm_reg1} and its
proof is almost identical.
Nevertheless, it is worth commenting on the fact that $\varepsilon_2$ does not appear in the statement of 
Theorem \ref{thm_reg3}. In the regime $\varepsilon_2 \ll \varepsilon_1$, 
the perturbation in the first order term is a regular perturbation, and as such benigne. As a result, its lack 
of presence in Theoreom \ref{thm_reg3} is not an issue.

\begin{theorem}
\label{thm_reg3}Assume (\ref{analytic}), (\ref{data}) hold. Then there exist
positive constants $K_{1},\check{K},\grave{K},K_{2}$ and $\delta $,
independent of $\varepsilon _{1},\varepsilon _{2},$ such that the solution $%
u $ of (\ref{de})--(\ref{bc}) can be decomposed as in (\ref{decomp3}), with%
\begin{equation*}
\left\Vert u_{M}^{(n)}\right\Vert _{\infty ,I}\lesssim n!K_{1}^{n}\;\; \forall
\;n\in \mathbb{N}_{0},
\end{equation*}%
\begin{equation*}
\left\vert \left( \check{u}_{M}^{BL}\right) ^{(n)}(x)\right\vert \lesssim 
\check{K}^{n}\varepsilon _{1}^{-n/2}e^{-\beta x/\sqrt{\varepsilon_{1}}}\quad \forall \;n\in \mathbb{N}_{0},
\end{equation*}%
\begin{equation*}
\left\vert \left( \grave{u}_{M}^{BL}\right) ^{(n)}(x)\right\vert \lesssim 
\grave{K}^{n}\varepsilon _{1}^{-n/2}e^{-\beta (1-x)/\sqrt{\varepsilon
_{1}}}\quad \forall \;n\in \mathbb{N}_{0},
\end{equation*}%
\begin{equation*}
\left\Vert r_{M}\right\Vert _{\infty ,\partial I}+\left\Vert
r_{M}\right\Vert _{0,I}+\varepsilon _{1}^{1/2}\left\Vert r_{M}^{\prime
}\right\Vert _{0,I}\lesssim e^{-\delta /\sqrt{\varepsilon _{1}}},
\end{equation*}%
where $M$ is chosen so that $\sqrt{\varepsilon _{1}}K_{2}M<1$, and $\beta = \min \left\{ \check{c}_{0}, \grave{c}_{0} \right\}$.
\end{theorem}

\subsection{On the transition between regimes}

As a final question, we would like to see what happens when we fix $\varepsilon_1 \ll 1$ and consider
$\varepsilon_2 \in (0, 1]$. In Figure \ref{fig1} we show the solution of the BVP
\begin{eqnarray*}
-0.005 u''(x) + \varepsilon_2 u'(x) + u(x) &=& 1, \: x \in (0,1) \\
u(0) = u(1) &=& 0,
\end{eqnarray*}
for $\varepsilon_2 \in (0,1]$.
\begin{figure}[h]
\begin{center}
\includegraphics[width=0.475 \textwidth]{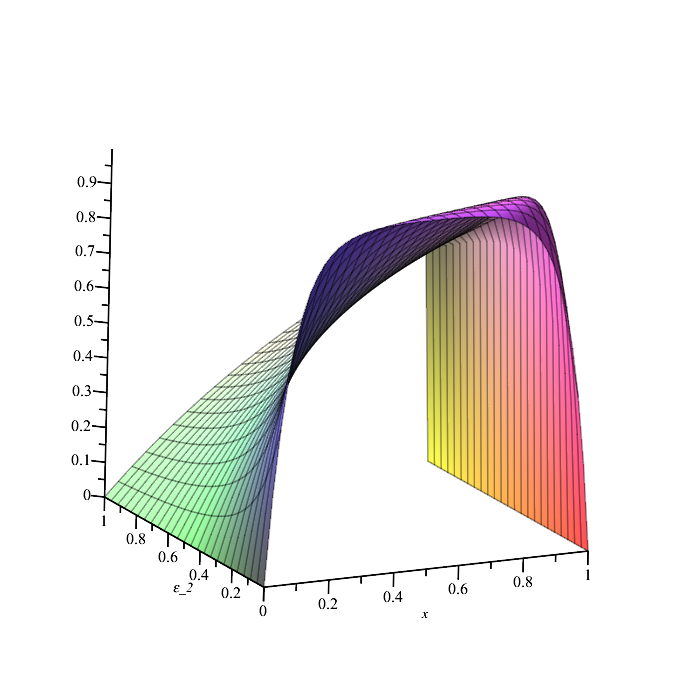}\mbox{}
\includegraphics[width=0.475 \textwidth]{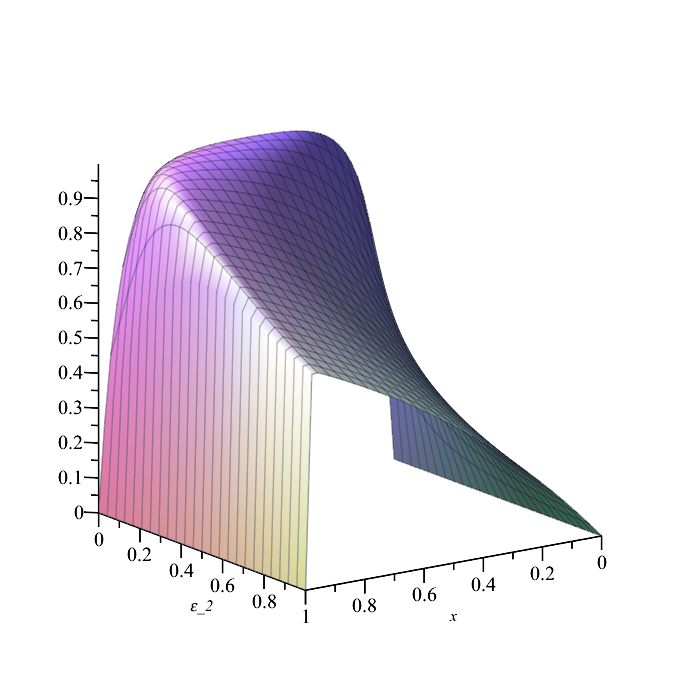}
\end{center}
\caption{The solution $u(x)$ as a function of $x$ and $\varepsilon_2$ (different viewing angles).}
\label{fig1}
\end{figure}
The figure shows that the trasition between regimes appears ``seamless'', 
in the following sense: as $\varepsilon_2$ takes on values in $(0,1]$,
the solution $u$ ``smoothly moves''  from one regime to the other, based on the relationship between $\varepsilon_2$
and (the fixed, but small) $\varepsilon_1$. In particular:
\begin{itemize}
\item If $\varepsilon_2=1$, then we have a convection-diffusion problem, and we have a layer
of width $O(\varepsilon_1)$ at the outflow boundary. (This is clearly visible in Figure \ref{fig1}, on the right.)

\item If $1 > \varepsilon_2 > \sqrt{\varepsilon_1}$, then we have a reaction-convection-diffusion problem,
with layers of different width at each endpoint.

\item If $0 < \varepsilon_2 \leq \sqrt{\varepsilon_1}$, then we have a reaction-diffusion problem, with layers
of width $O(\sqrt{\varepsilon_1})$ at each endpoint. (See Figure \ref{fig1}, on the left.)
\end{itemize}

The error bounds of the previous sections allow us to state the following:

\begin{proposition}
Let $\varepsilon_1 \ll 1$ be fixed, and let $u$ be the solution of (\ref{de})--(\ref{bc}) under the 
assumption (\ref{analytic}). Then for any $\varepsilon_2 \in (0, 1)$, there exist positive constants
$K_1, K_2, K_3, \delta$, independent of $\varepsilon_1, \varepsilon_2$, such that
\[
u = u_S + u^{\pm}_{BL} + u_R.
\]
The smooth part $u_S$, satisfies for any $n \in \mathbb{N}$ , $\varepsilon_2 \in (0,1)$,
\[
\left\Vert u_S^{(n)} \right\Vert_{\infty, I} \lesssim K_1^n n!.
\]
The boundary layer parts $u^{\pm}_{BL}$,  satisfy for any $n \in \mathbb{N}$ and 
\begin{itemize}
\item $\forall \; \; \varepsilon_2 \in (0,\sqrt{\varepsilon_1}],  \; x \in \overline{I}$,
\[
\left\vert (u_{BL}^{\pm})^{(n)} (x)\right\vert \lesssim K_2^n n! \varepsilon_1^{-n/2} 
e^{- \beta \text{ dist } (x, \partial I) / \sqrt{\varepsilon_1}} ,
\]
\item $\forall \; \; \varepsilon_2 \in (\sqrt{\varepsilon_1},1) , \; x \in \overline{I}$,
\[
\left\vert (u_{BL}^{-})^{(n)}(x) \right\vert \lesssim K_2^n n! \varepsilon_2^{-n} 
e^{- \beta x / \varepsilon_2 }  \; , \;  \left\vert (u_{BL}^{+})^{(n)} \right\vert \lesssim K_2^n n! 
\left( \frac{\varepsilon_1}{\varepsilon_2}\right)^{-n} e^{- \beta (1-x) \varepsilon_2 / \varepsilon_1}.
\]
\end{itemize}
The remainder $u_R$, satisfies for any $n \in \mathbb{N}$, $\varepsilon_2 \in (0,1)$,
\[
\Vert u_R \Vert_{\infty, \partial I} + \Vert u_R \Vert_{0, I} +  
\min\{ \varepsilon_2 , \sqrt{\varepsilon_1} \} \Vert u'_R \Vert_{0, I} \lesssim
\max\{ e^{ - \delta \varepsilon_2 / \varepsilon_1},  e^{ - \delta  / \varepsilon_2}\}.
\]

\end{proposition}

\section{Conclusions}

\label{concl} We considered a two-point, singularly perturbed,
reaction-convection-diffusion problem with analytic input data, and we derived
regularity results for its solution. Based on the relationship between the
singular perturbation parameters, the problem becomes convection-diffusion,
reaction-diffusion or reaction-convection-diffusion, as shown in Table 1. We
provided estimates for all three cases (regimes), which reveal the analytic
nature of the solution and give derivative bounds which are explicit in the
differentiation order as well as the singular perturbation parameters. Such
estimates are necessary for the construction and analysis of high order
numerical methods, such as $hp$ FEM (see, e.g. \cite{SX}).



\end{document}